\pdfoutput=1
\documentclass[a4paper,10pt]{amsart}
 \usepackage{
  amsmath,   amscd,    amssymb,   mathrsfs,
  mathtools, fullpage, enumerate, thmtools,
  fancyhdr,  stmaryrd, tikz-cd,   etoolbox,
  chngcntr,  yfonts, hieroglf}
\usepackage[shortlabels]{enumitem}
\usepackage[utf8]{inputenc}

\setlist{leftmargin=2\parindent}

\usepackage[pagebackref, hypertexnames=false, bookmarks=false]{hyperref}

\usepackage[capitalize]{cleveref}

\newenvironment{smatrix}{\left( \begin{smallmatrix} } {\end{smallmatrix} \right) }

\theoremstyle{plain}
\newtheorem{theorem}{Theorem}[section]

\newtheorem{lemma}[theorem]{Lemma}
\newtheorem{proposition}[theorem]{Proposition}
\newtheorem{corollary}[theorem]{Corollary}
\newtheorem{definition}[theorem]{Definition}
\newtheorem{notation}[theorem]{Notation}
\newtheorem{conjecture}[theorem]{Conjecture}
\newtheorem{assumption}[theorem]{Assumption}

\theoremstyle{remark}
\declaretheorem[name=Remark,sibling=theorem,qed={\lower-0.3ex\hbox{$\diamond$}}]{remark}
\declaretheorem[name=Note,sibling=theorem,qed={\lower-0.3ex\hbox{$\diamond$}}]{note}

\DeclareMathOperator{\Fil}{Fil}
\DeclareMathOperator{\GL}{GL}

\DeclareMathOperator{\GSp}{GSp}

\DeclareMathOperator{\Gal}{Gal}

\DeclareMathOperator{\Hom}{Hom}

\DeclareMathOperator{\Frob}{Frob}

\DeclareMathOperator{\ind}{ind}
\DeclareMathOperator{\Iw}{Iw}

\DeclareMathOperator{\loc}{loc}
\DeclareMathOperator{\tors}{tors}

\newcommand{\f}{\mathrm{f}}

\renewcommand{\AA}{\mathbf{A}}
\newcommand{\an}{\mathrm{an}}
\newcommand{\ans}{\an, \Sigma}

\newcommand{\CC}{\mathbf{C}}
\newcommand{\NN}{\mathbf{N}}
\newcommand{\DD}{\mathbf{D}}

\newcommand{\BB}{\mathbf{B}}

\newcommand{\cris}{\mathrm{cris}}
\newcommand{\Dcris}{\DD_{\cris}}
\newcommand{\DdR}{\DD_{\mathrm{dR}}}
\newcommand{\Drig}{\DD^\dag_{\rig}}

\newcommand{\QQbar}{\overline{\QQ}}
\newcommand{\QQ}{\mathbf{Q}}

\newcommand{\Ql}{\QQ_\ell}
\newcommand{\Qp}{\QQ_p}

\newcommand{\ZZ}{\mathbf{Z}}
\newcommand{\Zp}{\ZZ_p}

\newcommand{\cD}{\mathcal{D}}
\newcommand{\cE}{\mathcal{E}}

\newcommand{\cF}{\mathcal{F}}
\newcommand{\cG}{\mathcal{G}}
\newcommand{\cH}{\mathcal{H}}
\newcommand{\cK}{\mathcal{K}}
\newcommand{\cL}{\mathcal{L}}
\newcommand{\cM}{\mathcal{M}}
\newcommand{\cN}{\mathcal{N}}
\newcommand{\cO}{\mathcal{O}}
\newcommand{\cP}{\mathcal{P}}
\newcommand{\cR}{\mathcal{R}}

\newcommand{\cU}{\mathcal{U}}

\newcommand{\dR}{\mathrm{dR}}

\newcommand{\fP}{\mathfrak{P}}

\newcommand{\into}{\hookrightarrow}

\newcommand{\onto}{\twoheadrightarrow}
\newcommand{\ord}{\mathrm{ord}}
\newcommand{\res}{\mathrm{res}}
\newcommand{\rig}{\mathrm{rig}}

\newcommand{\pp}{\mathfrak{p}}

\newcommand{\Ki}{\cK_\infty}
\newcommand{\Kh}{\widehat{\cK}_\infty}

\numberwithin{equation}{section}
\renewcommand{\le}{\leqslant}

\renewcommand{\ge}{\geqslant}
\renewcommand{\geq}{\geqslant}

\makeatletter
\renewcommand\part{%
   \if@noskipsec \leavevmode \fi
   \par
   \addvspace{4ex}%
   \@afterindentfalse
   \secdef\@part\@spart}

\def\@part[#1]#2{%
    \ifnum \c@secnumdepth >\m@ne
      \refstepcounter{part}%
      \addcontentsline{toc}{part}{\thepart\hspace{1em}#1}%
    \else
      \addcontentsline{toc}{part}{#1}%
    \fi
    {\parindent \z@ \raggedright
     \interlinepenalty \@M
     \normalfont
     \ifnum \c@secnumdepth >\m@ne
       \Large\bfseries \partname\nobreakspace\thepart
       \par\nobreak
     \fi
     \huge \bfseries #2%
     \par}%
    \nobreak
    \vskip 3ex
    \@afterheading}
\def\@spart#1{%
    {\parindent \z@ \raggedright
     \interlinepenalty \@M
     \normalfont
     \huge \bfseries #1\par}%
     \nobreak
     \vskip 3ex
     \@afterheading}
\makeatother

\author{David Loeffler}
\address[Loeffler]{UniDistance Suisse, Schinerstrasse 18, 3900 Brig, Switzerland}
\email{david.loeffler@unidistance.ch}
\urladdr{\href{http://orcid.org/0000-0001-9069-1877}{0000-0001-9069-1877}}

\author{Sarah Livia Zerbes}
\address[Zerbes]{Department of Mathematics, ETH Z\"urich, R\"amistrasse 101, 8092 Z\"urich, Switzerland}
\email{sarah.zerbes@math.ethz.ch}
\urladdr{\href{http://orcid.org/0000-0001-8650-9622}{0000-0001-8650-9622}}

\thanks{D.L. gratefully acknowledges the support of the European Research Council through the Horizon 2020 Excellent Science programme (Consolidator Grant ``ShimBSD: Shimura varieties and the BSD conjecture'', grant ID 101001051)}
\begin{document}

\title{Ultra-Kolyvagin systems and non-ordinary Selmer groups}

\begin{abstract}
 We develop a machine for bounding Selmer groups of Galois representations via Euler systems in ``non-ordinary'' settings, using Pottharst's definition of Selmer groups via Robba-ring $(\varphi, \Gamma)$-modules. Our approach relies on Sweeting's interpretation of Kolyvagin derivative classes via non-principal ultrafilters. We apply these results to prove new cases of the cyclotomic Iwasawa main conjecture for non-ordinary Rankin--Selberg convolutions.
\end{abstract}

\maketitle


\section{Introduction}

 \subsection{The problem} The theory of \emph{Euler systems} has been conspicuously successful in studying the relations between special values of $L$-functions and arithmetic, and hence in proving cases of some of the major open problems of mathematics, such as the BSD conjecture, Bloch--Kato conjecture, and Iwasawa main conjecture. 
 
 However, in the application of Euler systems to Iwasawa theory, an important stumbling block arises. Euler systems work very well for bounding Selmer groups of $p$-adic Galois representations defined by \emph{strict} local conditions at $p$; or more generally local conditions of \emph{Greenberg type}, in which the local condition is given by the orthogonal complement of a subrepresentation stable under the decomposition group at $p$. For applications to the Iwasawa main conjecture, one needs this subrepresentation to be of a specific size: for global Galois representations arising from cohomological, cuspidal automorphic representations of $\GL_n$ with $n$ even, the required subrepresentation should have dimension $\tfrac{n}{2} + 1$ (serving as a ``stepping stone'' to the Selmer group appearing in the main conjecture itself, which should be defined by a subrepresentation of dimension $\frac{n}{2}$).
 
 The existence of such subrepresentations is a rather restrictive condition: it frequently occurs that no such subrepresentation exists, e.g. if the Galois representation is locally irreducible at $p$, and these cases cannot be handled in the framework of Greenberg Selmer groups. The highly influential work of Pottharst \cite{pottharst13} allows one to define and study a much more general class of Selmer groups, using submodules of $(\varphi, \Gamma)$-modules in place of local subrepresentations. This allows a robust and natural \emph{formulation} of the Iwasawa main conjecture, applying to any Galois representation which is crystalline at $p$ (and many that are not). However, it has not previously been possible to use Euler systems to bound Pottharst's Selmer groups. That is the main goal of this article.
 
 \subsection{Local conditions and Kolyvagin systems} As input to our construction, we assume we have an Euler system which satisfies a non-trivial Pottharst-type local condition; we would like to obtain from this a bound for the Selmer group with that local condition.

 However, the actual bounds for Selmer groups are obtained by using an intermediate object, a \emph{Kolyvagin system}, obtained from the Euler system via Kolyvagin's ``derivative'' construction. This raises the problem: \emph{if an Euler system satisfies some non-trivial local condition, is this reflected in additional local properties of the Kolyvagin system derived from it?} This question is less straightforward than it might seem, for the following fundamental reason: Kolyvagin classes are ``mod $p^n$'' objects, taking values in the mod $p^n$ reduction of a Galois representation; whereas Pottharst's Selmer groups are ``$\Qp$-linear'' objects, with no natural integral structures. In order to bridge this gap we crucially use an insight originally due to Sweeting \cite{sweeting20}, which is that Kolyvagin classes can be ``patched together'' using \emph{non-principal ultrafilters} in order to give characteristic 0 limiting objects. These are the ``ultra-Kolyvagin systems'' of the paper's title. Assuming that the local condition satisfied by the Euler system is `closed' (in the sense of Conjecture \ref{conj:closed}), we can show that it propagates to the Kolyvagin system (Theorem \ref{thm:ultrabeef}), which allows us to use the Euler system machine to directly obtain a bound of Selmer groups with the orthogonal local condition. 
 
 Checking the `closedness' of local conditions which arise naturally is a delicate problem in $p$-adic Hodge theory. We check this property for Bloch--Kato ``$H^1_\f$'' and ``$H^1_{\mathrm{g}}$'' local conditions (Corollary \ref{h1fclosed}); and for Pottharst local conditions arising from (potentially non-\'etale) sub-$(\varphi,\Gamma)$-modules we prove in Corollary \ref{pottharstclosed} a slightly weaker result that turns out to be sufficient for Iwasawa-theoretic applications.

 \subsection{Applications}
 
 As a first application of our machinery, we consider the Euler system of \cite{loefflerzerbes16} for the Rankin--Selberg convolution of two non-ordinary modular forms. For this case, the strongest result avalable with previous methods was to show that the \emph{support} of the Selmer group in the Iwasawa main conjecture is contained in the vanishing locus of the $p$-adic $L$-function. With our methods we are able to upgrade this to an actual divisibility, thus proving the ``Euler system half'' of the Iwasawa main conjecture in this context. Similar arguments can be applied to numerous other Euler systems (e.g.~for $\GSp_4$, or for quadratic Hilbert modular forms); details will appear in subsequent works.

 \subsection*{Acknowledgements} We are very grateful to Laurent Berger, Pierre Colmez, Ananyo Kazi, Chris Skinner, Rustam Steingart, and Naomi Sweeting for illuminating conversations on the topics of this paper.

 \subsection*{Notation} We fix throughout this paper a prime $p > 2$, and a finite extension $E / \Qp$ with ring of integers $\cO$, uniformizer $\varpi$ and residue field $\mathbf{k}$; we will study Galois representations on $E$-vector spaces and $\cO$-modules.

\part{Local theory}

 We begin with some local preliminaries. The main results of this section are to establish that, for $F$ a ``big'' unramified extension of $\Qp$ (not of finite degree  over $\Qp$), and $V$ a suitable Galois representation, certain natural subspaces of $H^1(F, V)$ are closed in the Banach-space topology: see \cref{pottharstclosed} for subspaces defined using $(\varphi, \Gamma)$-modules ``\`a la Pottharst'', and \cref{h1fclosed} for the Bloch--Kato $H^1_\f$ and $H^1_{\mathrm{g}}$ subspaces.

\section{Local cohomology groups}

 Let $\Ki$ the unique unramified $\Zp$-extension of $\Qp$. Note that the completion $\Kh$ is a complete discretely valued field, with the same residue field as $\Ki$, and the groups $\Gal\left(\overline{\Kh} / \Kh\right)$ and $\Gal(\overline{\Ki} / \Ki)$ are canonically isomorphic; we identify them henceforth. We let $F$ be a finite unramified extension of $\Kh$.

 \subsection{Cohomology of Galois representations}

   Let $T$ be a finite-rank free $\cO$-module with a continuous action of $\Gal(\bar{F} / F)$, and let $V = T[1/p]$.

  \begin{proposition}\label{prop:banachH1} \
   \begin{itemize}
    \item For all $i \ge 2$, the groups $H^i(F, T)$ and $H^i(F, V)$ are zero, as is $H^i(F, T/\varpi^j)$ for all $j \ge 1$.

    \item The cohomology group $H^1(F, T)$ has finite $p$-torsion subgroup; and the quotient $H^1(F, T) / \{\mathrm{torsion}\}$ is a flat, $\varpi$-adically complete and separated $\cO$-module, whose base-extension to $E$ is canonically identified with $H^1(F, V)$.
   \end{itemize}
  \end{proposition}

  \begin{proof}
   By a general result of Tate \cite[Proposition 2.2]{tate76}, we have short exact sequences
   \[ 0 \to {\varprojlim_j}^{(1)} H^{i-1}(F, T/\varpi^j) \to H^i(F, T) \to \varprojlim_j H^i(F, T/\varpi^j) \to 0.\]
   We claim that the groups $H^i(F, T/\varpi^j)$ are zero for $i \ge 2$ and all $j$. This is because $F$ is a complete discretely valued field whose residue field $\mathbf{k}_{\Kh}$ has $p$-cohomological dimension zero, since its Galois group is an inverse limit of finite groups of prime-to-$p$ order. By Theorem 6.5.15 of \cite{nsw}, we have $\operatorname{cd}_p(\Kh) = 1$; that is, $H^i(\Kh, A) = H^i(F, A)$ vanishes for all $p$-torsion modules $A$ and all $i \ge 2$.

   It follows that the inverse system $\left( H^{i-1}(F, T/\varpi^j)\right)_{j \ge 1}$ satisfies the Mittag--Leffler condition. For $i = 0, 1$ this is obvious (since the groups concerned are finite); and for $i \ge 2$, the cokernel of the map $H^{i-1}(F, T/\varpi^{j + 1}) \to H^{i-1}(F, T/\varpi^j)$ injects into $H^i(F, T/\varpi)$, which is zero by the above, so the transition maps are surjective. Hence the map $H^i(F, T) \to \varprojlim_j H^i(F, T/\varpi^j)$ is an isomorphism for all $i$, and in particular $H^i(F, T) = 0$ for $i \ge 2$. Finally, since $\Gal(\QQbar_p / F)$ is a compact group, any continuous $V$-valued cochain must factor through $\varpi^{-m} T \subseteq V$ for some $m$, and so lies in the image of $H^i(F, \varpi^{-m} T) \cong H^i(F, T) = 0$. This completes the proof of the first point.

   For the second point, we observe that the torsion in $H^1(F, T)$ is identified with the image of $H^0(F, V / T)$ modulo its divisible part, which is a finite group. Moreover, since $H^2$ vanishes, we have $H^1(F, T) / \varpi^j \cong H^1(F, T/\varpi^j)$, so we conclude that $H^1(F, T) \cong \varprojlim_j H^1(F, T) / \varpi^j$ is $\varpi$-adically complete and separated.
  \end{proof}

  It follows that $H^1(F, V)$ is naturally an $E$-Banach space, with $H^1(F, T) / \{\mathrm{torsion}\}$ as its unit ball. It is, however, not generally finite-dimensional over $E$.

 \subsection{Cohomology of $(\varphi, \Gamma)$-modules}

  We consider the usual rings of $(\varphi, \Gamma)$-module theory
 \begin{subequations}
 \begin{equation}
   \label{eq:rings}
   \AA_F = \cO_F[[\pi]][1/\pi]^\wedge, \qquad \BB_F = \AA_F[1/p], \qquad \AA^{\dag}_F,\qquad \BB^{\dag}_F, \quad \BB^{\dag}_{\rig, F},
  \end{equation}
  and their subrings
  \begin{equation}
   \AA^+_F = \cO_F[[\pi]], \qquad \BB^+_F = \AA^+_F[1/p], \qquad \BB^+_{\rig, F}.
  \end{equation}
  \end{subequations}

  \begin{remark}
   The notations above are those of Colmez (see e.g.~\cite{cherbonniercolmez98}); but we shall also denote $\BB^{\dag}_{\rig, F}$ (the Robba ring) by the more familiar notation $\cR_F$, and $\BB^+_{\rig, F}$ by $\cR_F^+$. (The other rings above will not occur outside this subsection.)
  \end{remark}

  For a $G_F$-representation $V$ we have $(\varphi, \Gamma)$-modules $\DD(V)$, $\DD^\dag(V)$, $\Drig(V)$ over $\BB_F$, $\BB^\dag_F$, $\BB^\dag_{\rig, F}$ respectively; and for $T \subset V$ a lattice we have integral modules $\DD(T)$, $\DD^\dag(T)$ over $\AA_F$, $\AA^\dag_F$ respectively.

  \begin{proposition}
   If $F$ is a finite extension of $\Kh$, and $T$ is a lattice in a $G_F$-representation $V$, then the map $\psi - 1$ is surjective on all of the modules $\DD(T)$, $\DD(V)$, $\DD^{\dag}(T)$, $\DD^{\dag}(V)$, and $\Drig(V)$.
  \end{proposition}

  \begin{proof}
   (The following argument is due to Berger and Steingart, and we are grateful to them for sharing it with us.) First we claim that $\psi - 1$ is surjective on the Fontaine $(\varphi, \Gamma)$-module $\DD(T)$. This is proved in \cite[Lemma VI.7]{berger-thesis} assuming the residue field to be algebraically closed; and since $\Gal(\overline{\mathbf{k}}_F / \mathbf{k}_F)$ is an inverse limit of finite groups of prime-to-$p$ order, we can descend to $F$.

   We now apply the first half of the proof of Lemma 2.6 in \cite{liu08} (which does not require $[F : \Qp]$ to be finite, although the second half does) to conclude that $\DD^\dag(T) / (\psi - 1) \DD^{\dag}(T)$ injects into $\DD(T) / (\psi - 1) \DD(T)$, and hence must be zero. So we have proved the claim for $\DD(T)$ and $\DD^{\dag}(T)$, and the result for $\DD(V)$ and $\DD^{\dag}(V)$ follows.

   To conclude, we use \cite[Proposition 1.5.4]{kedlaya08}, which shows that the natural map $\DD^{\dag}(V) / (\varphi - 1) \to \Drig(V) / (\varphi - 1)$ is an isomorphism (note this concerns $\varphi$, not $\psi$). However, since $(\varphi - 1) \cD = (\psi - 1)(\varphi(\cD)) \subseteq (\psi - 1) \cD$, this implies $\cD / (\psi - 1) = 0$.
  \end{proof}

  \begin{definition}
   For $\cD$ a $(\varphi, \Gamma)$-module over any of the above rings, we write $C_{\varphi, \Gamma}(\cD)$ for its Herr complex (where $\gamma$ is a topological generator of $\Gamma$)
   \[ [\cD \xrightarrow{\left(\begin{smallmatrix}\varphi-1 \\ \gamma-1\end{smallmatrix}\right)} \cD \oplus \cD \xrightarrow{(1 - \gamma, \varphi - 1)} \cD], \]
   and $C_{\psi, \Gamma}(\cD)$ for the analogous complex with $\psi$ instead of $\varphi$.
  \end{definition}

  There is a natural surjective map of complexes $C_{\varphi, \gamma}(\cD) \to C_{\psi, \gamma}(\cD)$ whose kernel is the complex $[0 \longrightarrow \cD^{\psi = 0} \xrightarrow{\gamma - 1}\cD^{\psi = 0}]$.

  \begin{proposition}[Herr, Cherbonnier--Colmez, Kedlaya--Pottharst--Xiao, Liu]
   If $V$ is a $G_F$-representation and $\cD = \DD(V), \DD^\dag(V), \DD^{\dag}_\rig(V)$, or $T \subset V$ is a lattice and $\cD = \DD(T), \DD^\dag(T)$, then the complexes $C_{\varphi, \gamma}(\cD)$ and $C_{\psi, \gamma}(\cD)$ are quasi-isomorphic and their cohomology groups agree with the Galois cohomology groups $H^i(F, -)$.
  \end{proposition}

  \begin{proof}
   For $F / \Qp$ finite, this is shown in \cite[\S 2]{liu08}. For $F / \cK_\infty$ finite we must be a little more circumspect. For $\cD = \DD(T)$, the isomorphisms  $H^i(C_{\varphi, \gamma}(\cD)) \cong H^i(C_{\psi, \gamma}(\cD)) \cong H^i(F, T)$ are shown in \cite{herr98} (without assuming $\mathbf{k}_F$ finite).

   For $\cD = \DD^\dag(T)$, the fact that $\gamma - 1$ is surjective on $ \cD^{\psi = 0}$ is shown in \cite{cherbonniercolmez98}; so the $\varphi$ and $\psi$ complexes are quasi-isomorphic in this case as well. We have just seen that $\cD / (\psi - 1) = 0$, so the $H^2$ is zero; moreover, $H^1(C_{\varphi, \gamma}(\cD))$ clearly classifies extensions of the trivial object by $\cD$ in the category of \'etale $\varphi$-modules over $\AA^\dag_F$, which is equivalent to the category of $\Zp$-representations of $G_F$ by the main result of \cite{cherbonniercolmez99}, so it agrees with Galois $H^1$.

   Finally we consider $\cD = \Drig(V)$. The bijectivity of $\gamma - 1$ on $\cD^{\psi = 0}$ can be proved as in \cite[\S 3.1]{KPX}: although $F$ is assumed to be finite over $\Qp$ in \emph{op.cit.}, and the argument proceeds by reduction to the case $F = \Qp$, in fact the argument given there for $F = \Qp$ extends without change to any complete unramified field (not necessarily finite over $\Qp$). With this in hand, the argument proceeds as before.
  \end{proof}

 \subsection{Slope $\ne 0$ case}

  Over the ring $\cR_F$ there are interesting $(\varphi, \Gamma)$-modules which are not \'etale (equivalently, not pure of slope 0, in the sense of Kedlaya's slope filtration theory \cite{kedlaya08}). So we shall investigate the cohomology of these a little more. The argument quoted above from \cite{KPX} does not require $\cD$ to be \'etale, so we have the following:

  \begin{corollary}
   For any $(\varphi, \Gamma)$-module $\cD$ over $\cR_F$, the map $\gamma - 1$ is bijective on $\cD^{\psi = 0}$. Hence the complexes $C_{\varphi, \gamma}(\cD)$ and $C_{\psi, \gamma}(\cD)$ have the same cohomology.\qed
  \end{corollary}

  \begin{notation}
   We denote the common cohomology by $H^*(\cD)$; as we have seen, this agrees with $H^*(F, V)$ when $\cD= \Drig(V)$. We denote by $H^1(\cD)^+$ the submodule $\cD^{\psi = 1} / (\gamma - 1)$ of $H^1(\cD)$, so that $H^1(\cD) / H^1(\cD)^+ = (\cD / (\psi - 1)\cD)^{\Gamma}$.
  \end{notation}

  \begin{conjecture}
   We have $\cD / (\psi - 1) = 0$ for all $(\varphi, \Gamma)$-modules $\cD$ over $\cR_F$.
  \end{conjecture}

  This of course implies that $H^2(\cD) = 0$ for all $\cD$, and $H^1(\cD) = H^1(\cD)^+$. We have seen that the conjecture holds when $\cD$ is pure of slope 0; hence it also holds for $\cD$ that are quotients of pure slope 0 modules (necessarily of slope $\le 0$).

  \begin{remark}
   For $L/\Qp$ finite, one can show that any $(\varphi, \Gamma)$-module over $\cR_L$ of slope $\le 0$ is a quotient of one of slope 0 (see \cite{liu08} for example). If this is also true over $F$, then it follows that the conjecture holds for all $\cD$ of slope $\le 0$. However, the proof for $L / \Qp$ finite relies on Tate's local Euler characteristic formula to produce appropriate non-split extensions, so the argument does not immediately extend to the present setting.
  \end{remark}

 \subsection{Conjecture: canonical topologies}

  We formulate the following:

  \begin{conjecture}[Optimistic Conjecture]
   For $\cD$ a $(\varphi, \Gamma)$-module over $\cR_F$, the group $H^1(\cD)$ carries a natural Banach-space topology which is functorial in $\cD$, and coincides with the Banach-space structure on $H^1(\Ki, V)$ from \cref{prop:banachH1} when $\cD = \Drig(V)$.
  \end{conjecture}

  We have not been able to prove this conjecture in its full generality. Hence we formulate a much weaker conjecture, which is actually sufficient for our purposes:

  \begin{conjecture}[Pessimistic Conjecture]
   \label{conj:closed}
   If $V$ is a $\Gal(\QQbar_p / \Ki)$-representation, and $\cD^+$ a saturated $\cR$-submodule of $\cD = \Drig(V)$, then the image of the map $H^1(\cD^+) \to H^1(\cD) \cong H^1(\Ki, V)$ is closed in the Banach-space topology of $H^1(\Ki, V)$.
  \end{conjecture}

  We shall prove a very special case of this below.


\section{Regulator maps}

 As before $F$ denotes a finite unramified extension of $\Kh$.

 \subsection{$\varphi$-modules}

  By a $\varphi$-module over $F$, we mean a finite-dimensional $F$-vector space $D_0$ with a bijective semilinear Frobenius $\varphi$ (acting as the arithmetic Frobenius on $F$).

  \begin{proposition}
   \label{prop:Khproperties0}
   If $D_0$ is any $\varphi$-module over $F$, then $(1 - \varphi) D_0 = D_0$.
  \end{proposition}

  \begin{proof}
   Since $D_0$ has a slope decomposition by the Dieudonn\'e--Manin theorem, we can assume $D_0$ is isoclinic (has a single slope). For non-zero slopes the result is clear, since either $\sum_{n \ge 0} \varphi^n$ or $-\sum_{n \ge 0} \varphi^{-1-n}$ converges to an inverse of $1 - \varphi$. So we may assume $D_0$ is the base-extension of a $\varphi$-module over $\cO_F$, and by d\'evissage we reduce to the case of $\varphi$-modules over the residue field $\mathbf{k}_F$.

   After base-extension to $\overline{\mathbf{k}}_F$, there exists a $\varphi$-invariant basis by Hilbert's theorem 90; and since solving $x - x^p = y$ for a given $y \in \overline{\mathbf{k}}_F$ amounts to solving a separable polynomial, we see that $1 - \varphi$ is surjective on $\overline{\mathbf{k}}_F^{\oplus d}$. Thus for the original field $\mathbf{k}_F$, we can always find solutions after a finite extension of $\mathbf{k}_F$; but this extension necessarily has prime-to-$p$ degree, so we obtain a solution over $\mathbf{k}_F$ by taking the trace.
  \end{proof}

  We now study the modules $\cR^+_F \otimes_F D_0$ and $\cR_F \otimes_F D_0$.
 %
 %

  \begin{proposition}[Chenevier, Nakamura]
   \label{prop:Khproperties1}
   Let $D_0$ be any $\varphi$-module over $F$.
   \begin{enumerate}[(i)]
    \item $1 - \varphi$ is surjective on $\cR^+_F \otimes D_0$, and its kernel is given by $\bigoplus_{k \ge 0} t^k \otimes D_0^{(\varphi = p^{-k})}$.

    \item $1 - \psi$ is surjective on $\cR_F^+ \otimes_F D_0$ and on $\cR_F\otimes_F D_0$.

    \item We have isomorphisms
    \begin{align*}
     \left(\cR_F \otimes_F D_0\right)^{\psi=1}\, /\,  \left(\cR_F^+ \otimes_F D_0\right)^{\psi=1}
     &\cong \bigoplus_{k \ge 0} x^{k} \otimes D_0^{(\varphi = p^k)},
    \end{align*}
    where $x^k \in \cR_F / \cR^+_F$ satisfies $\gamma(x) = \chi^{-(1+k)}(\gamma)\cdot x$.

   \end{enumerate}
  \end{proposition}

  \begin{proof}
   See Lemma 3.17 of \cite{Nakamura-Iwasawa} (generalizing \cite[Lemma 2.9]{chenevier13}) for an analogous statement for $\varphi$-modules over $\Qp$, in which various cokernel terms involving $D_0 / (1 - p^k \varphi) D_0$ for $k \in \ZZ$ intervene. In the present setting of $\varphi$-modules over finite extensions of $\Kh$, the same formulae remain valid (with the same proofs) but these cokernel terms are all zero by \cref{prop:Khproperties0}.
  \end{proof}

 \subsection{Crystalline $(\varphi, \Gamma)$-modules}

  We recall the following definitions (see e.g.~\cite[\S 2.7]{bellaichechenevier09}):

  \begin{definition}
  	A $(\varphi, \Gamma)$-module $\cD$ over $\cR_F$ is  \emph{crystalline} if $\cD[1/t]$ has an $\cR_F$-basis invariant under $\Gamma$. If this is the case, we define
   \begin{align*}
    \Dcris(\cD) &\coloneqq \cD[1/t]^{\Gamma}, \\
    \NN_{\cris}(\cD) &\coloneqq \cR_F \otimes_F \Dcris(\cD),\\
    \NN^+_{\cris}(\cD) &\coloneqq \cR^+_F \otimes_F \Dcris(\cD).
   \end{align*}
   Thus both $\NN_{\cris}(\cD)$ and $\NN^+_{\cris}(\cD)$ are submodules of $\cD[1/t]$ invariant under the actions of $\varphi$ and $\Gamma$; and $\Dcris(\cD)$ is an $F$-vector space of dimension equal to the $\cR_F$-rank of $\cD$ (with a semilinear action of $\varphi$).
  \end{definition}

  Suppose henceforth that $\cD$ is crystalline, so that $\Dcris(\cD)$ is a $\varphi$-module over $F$. We equip $\Dcris(\cD)$ with a filtration via the natural embedding into $F(\zeta_n)((t))$ for $n \gg 0$, as in \cite{bellaichechenevier09}; then, in accordance with the usual conventions the \emph{Hodge--Tate weights} are the negatives of the jumps in this filtration. With these definitions, if $\cD = \Drig(V)$, then $\DD_{\cris}(\cD)$ is canonically isomorphic (as a filtered $\varphi$-module) to $\DD_{\cris}(V)$ as defined using Fontaine's traditional period rings.

  \begin{proposition}
   If the Hodge--Tate weights of $\cD$ are in the interval $[a, b]$, then $t^b \NN_{\cris}(\cD) \subseteq \cD \subseteq t^a \NN_{\cris}(\cD)$.
  \end{proposition}

  \begin{proof}
   The assumption that the Hodge--Tate weights are $\le b$ means that $\Dcris(\cD) \subseteq t^{-b} \cD$, and hence $\NN_{\cris}(\cD) \subseteq t^{-b} \cD$. The reverse inequality follows by applying the above to $\cD^* = \Hom_{\cR_F}(\cD, \cR_F)$.
  \end{proof}
%
%
%
%
%

 \subsection{Regulator maps}

  Let us now suppose (until further notice) that $\cD$ has Hodge--Tate weights in $[0, b]$ for some $b \ge 0$, so $\cD \subset \NN_{\cris}(\cD)$ and hence $\cD^{\psi = 1} \subset \NN_\cris(\cD)^{\psi = 1}$.

  \begin{definition}
   We let $I_\cD$ be the invertible ideal of $\cH_{\Qp}(\Gamma)$ that is the support of the quotient
   \[ \frac{\cD^{\psi=1}}{\cD^{\psi = 1} \cap \NN^+_{\cris}(\cD)}. \]
  \end{definition}

  Note that $I_\cD$ is supported at a finite direct sum of (negative) powers of the cyclotomic character, by \cref{prop:Khproperties1}, since we have $\frac{\cD^{\psi=1}}{\cD^{\psi = 1} \cap \NN^+_{\cris}(\cD)} \ \subseteq\ \frac{\NN_{\cris}(\cD)^{\psi = 1}}{\NN^+_\cris(\cD)^{\psi = 1}}$.

  \begin{definition}
   The \emph{Perrin-Riou regulator} for $\cD$ is the map
   \[
    \cL_{\cD} : \cD^{\psi = 1} \cap \NN^+_{\cris}(\cD) \xrightarrow{(1 - \varphi)} \NN^+_{\cris}(\cD)^{\psi = 0} \xrightarrow{\ \cong\ } \cH_F(\Gamma) \otimes_F \Dcris(\cD),
   \]
   which we can extend in the obvious way to
   \[ \cL_{\cD} : \cD^{\psi = 1} \to I_{\cD}^{-1}\cH_F(\Gamma) \otimes_F \Dcris(\cD). \]
  \end{definition}

  See e.g.~\cite{leiloefflerzerbes11}; it is assumed in \emph{op.cit.} that $F / \Qp$ is a finite extension, but the theory works without change in the generality here.

  \begin{proposition}
   We have
   \[ \ker(\cL_\cD) = \bigoplus_{k \ge 0} t^k \otimes (\Dcris(\cD)^{\varphi = p^{-k}} \cap \Fil^{-k} \Dcris(\cD)), \]
   so the support of $\ker(\cL_\cD)$ consists of the characters $\chi^k$, for integers $k \ge 0$ such that $H^0(\cD(-k)) \ne 0$.
  \end{proposition}

  \begin{definition}\label{def:Lkappa}
   Let $\kappa$ be a character $\Gamma \to E^\times$ which is not in the support of the ideal $I_{\cD}$. Then we let
   \[ \cL_{\cD, \kappa} : H^1(\cD(\kappa^{-1}))^+ = \cD^{\psi = 1} / (\gamma - \kappa(\gamma)) \to \Dcris(\cD) \]
   denote the map induced by the composite of the Perrin-Riou regulator and evaluation at $\kappa$.
  \end{definition}

  \begin{remark}\label{rmk:twistcompat}
   Note that the maps $\cL_{\cD(1), \kappa}$ and $\cL_{\cD, \kappa\chi^{-1}}$ have the same source and target, but are not equal; rather, assuming both maps are defined, we have
   \[ \cL_{\cD(1), \kappa} = \kappa'(1) \cdot \cL_{\cD, \kappa\chi^{-1}}, \]
   so in particular $\cL_{\cD(1), \kappa} = 0$ for all locally-constant characters $\kappa$.
  \end{remark}

 \subsection{Bijectivity for almost all $\kappa$}

  We now show that for ``most'' characters $\kappa$, the map $\cL_{\cD, \kappa}$ is a bijection, assuming $F / \Kh$ is infinite. From \cref{prop:Khproperties1} we have the following:

  \begin{proposition}\label{prop:oneminusphisurj}
   The map
   \[
    1 - \varphi : \NN^+_{\cris}(\cD)^{\psi = 1} \longrightarrow \NN^+_{\cris}(\cD)^{\psi = 0} \cong \cH_F(\Gamma) \otimes_F \Dcris(\cD)
   \]
   is surjective.
  \end{proposition}

  \begin{notation}
   Let $\nabla \in \cH_{\Qp}(\Gamma)$ denote the ``differentiate at 1'' distribution, defined by $\frac{\log [\gamma]}{\log \chi(\gamma)}$ for any non-torsion $\gamma \in \Gamma$.
  \end{notation}

  \begin{proposition}
   The modules $\NN_\cris(\cD)^{\psi = 1} / \cD^{\psi = 1}$ and $\cD / (\psi - 1) \cD$ are annihilated by $\nabla (\nabla - 1) \dots (\nabla - b + 1) \in \cH_{\Qp}(\Gamma)$.
  \end{proposition}

  \begin{proof}
   From \cite[Theorem 3.5 (3)]{Nakamura-Iwasawa} we have the relation
   \[ (\nabla - k) (t^k \NN_{\cris}(\cD)) \subseteq t^{k + 1} \NN_{\cris}(\cD). \]
   Since $\cD^{\psi = 1}$ is sandwiched between $\NN_{\cris}(\cD)^{\psi = 1}$ and $t^b \NN_{\cris}(\cD)^{\psi = 1}$, the first formula follows. The second follows from this together with the fact that $(\psi - 1)\NN_\cris(\cD) = \NN_\cris(\cD)$.
  \end{proof}

  \begin{corollary}
   \label{cor:goodkappa}
   Let $\kappa$ be a character with the following properties:
   \begin{itemize}
   \item $\kappa'(1) \notin \{0, 1, \dots, b-1\}$,
   \item if $\kappa = \chi^{-(1 + k)}$ for $k \ge 0$, then $\Dcris(\cD)^{\varphi = p^k} = 0$ (so $\kappa$ is not in the support of $I_{\cD}$),
   \item if $\kappa = \chi^k$ for $k \ge 0$, then $H^0(\cD(-k)) = 0$ (so $\kappa$ is not in the support of $\ker \cL_\cD$).
   \end{itemize}
   Then $\cL_{\cD, \kappa}$ gives a bijection
   \[ H^1(\cD(\kappa^{-1})) \cong \Dcris(\cD).\]
   \vskip -\baselineskip

   \qed
  \end{corollary}

  \begin{remark}
   The corollary also holds as stated for $F / \Qp$ finite, although the proof is slightly different: it is no longer true that $1 - p^k \varphi$ is always surjective on $\Dcris(\cD)$, but (for dimension reasons) if it is injective then it is also surjective, and this is sufficient.
  \end{remark}

 \subsection{The regulator topology} We now recall that since $F$ is a complete non-discrete normed field, any finite-dimensional $F$-vector space has a canonical $F$-Banach-space topology (the topology coming from any choice of $F$-vector-space isomorphism with $F^{\oplus d}$). We now no longer assume that $\cD$ has Hodge--Tate weights $\ge 0$.
 
  \begin{definition}
   Assume $\cD$ is crystalline with Hodge--Tate weights in $[a, b]$, for some $a\le b\in \ZZ$; and let $\kappa$ be a character such that
   \begin{itemize}
    \item $\kappa'(1) \notin \{a, \dots, b-1\}$;
    \item if $\kappa = \chi^{-(1 + k)}$ for $k \ge -a$, then $\Dcris(\cD)^{\varphi = p^k} = 0$;
    \item if $\kappa = \chi^k$ for $k \ge a$, then $H^0(\cD(-k)) = 0$.
   \end{itemize}
   Then we define the \emph{regulator topology} on $H^1(\cD(\kappa^{-1}))$ to be the unique topology such that the regulator map
   \[ \cL_{\cD(-a), \kappa \chi^{-a}} : H^1(\cD(\kappa^{-1})) \to \Dcris(\cD), \]
   (which is a bijection by \cref{cor:goodkappa}) is a homeomorphism, for the canonical $F$-Banach topology on $\Dcris(\cD)$.
  \end{definition}
  
  This topology does not depend on the choice of $a$, because \cref{rmk:twistcompat} shows that the regulator maps for different choices differ by a nonzero scalar (and hence induce the same topology).
  
  \begin{corollary}\label{cor:canFtopfunctorial}
   Let $\cD_1 \to \cD_2$ be a morphism of $(\varphi, \Gamma)$-modules, with $\cD_1$ and $\cD_2$ crystalline.
   \begin{enumerate}[(i)]
    \item The set of $\kappa$ such that the regulator topologies are defined for both $(\cD_1, \kappa)$ and $(\cD_2, \kappa)$ is the complement of a discrete closed subset of the rigid-analytic space $\Hom(\Gamma, \mathbf{G}_m^{\mathrm{rig}})$ of characters of $\Gamma$. 
    \item For any $\kappa$ satisfying this condition, the map
    \[ H^1(\cD_1(\kappa^{-1})) \to H^1(\cD_2(\kappa^{-1}))\]
    is continuous and strict, with closed image, for the regulator topologies on both sides.
   \end{enumerate}
  \end{corollary}

  \begin{proof}
   The first statement is clear, since the second and third conditions on $\kappa$ only rule out finitely many characters, and the set of characters with $\kappa'(1) = m$ is discrete for any $m$. 
   
   For the second statement: we may assume without loss of generality that both of the $\cD_i$ have Hodge--Tate weights $\ge 0$. From the construction of the regulator map it is clear that we have a commutative diagram
   \[
    \begin{tikzcd}
     H^1(\cD_1(\kappa^{-1})) \rar \dar["\cL_{\cD_1, \kappa}" left] & H^1(\cD_2(\kappa^{-1}))
     \dar["\cL_{\cD_2, \kappa}" right] \\
     \Dcris(\cD_1(\kappa^{-1})) \rar & \Dcris(\cD_2(\kappa^{-1})
    \end{tikzcd}
   \]
   in which the map $\Dcris(\cD_1) \to \Dcris(\cD_2)$ is $F$-linear, and hence continuous and strict with closed image for the canonical $F$-Banach topology. By definition, the vertical maps are homeomorphisms. Since any $F$-linear map between finite-dimensional $F$-vector spaces is continuous and strict with closed image for the canonical topology, and the map $\Dcris(\cD_1) \to \Dcris(\cD_2)$ is $F$-linear, we are done.
  \end{proof}

  \begin{remark}
   Note that we do \emph{not} claim that the regulator topology on $H^1(\cD(\kappa^{-1}))$ coincides with the topology induced, via the Herr complex, from the natural topology on $\cD$ itself (which is a direct limit of Fr\'echet spaces). Of course it is natural to expect that these topologies should agree, but this does not seem to be straightforward to prove; and we shall not use the topology on $H^1(\cD(\kappa^{-1}))$ induced from that of $\cD$ in the present work.
  \end{remark}

\section{Integral theory}

 In this section we take $\cD = \DD^{\dag}_\rig(V)$ for a crystalline representation $V$ with Hodge--Tate weights in $[0, b]$, and show that the topology on $H^1(\cD(\kappa^{-1}))$ in \cref{cor:canFtopfunctorial} coincides with the Banach-space topology of $H^1(F, V(\kappa^{-1})))$. This is obvious for $F / \Qp$ finite, so we shall assume $F$ is a finite extension of $\Kh$.

 \subsection{Wach modules}

  Write $\NN(V)$ for the Wach module of $V$ in the sense of \cite{berger04}; and for $T \subset V$ a choice of $\cO$-lattice, write $\NN(T)$ for the corresponding Wach module, so that $\NN(V) = \NN(T)[1/p]$.

  Then $\NN(T)$ is a finite-rank $\cO_F[[\pi]]$-submodule of $\DD^{\dag}(T)$, free of rank $d \coloneqq \dim V$, with the following properties:
  \begin{enumerate}
   \item We have
   \( q^b \varphi^*\NN(T)\subseteq \NN(T)\subseteq\varphi^*\NN(T),\)
   where $q=\varphi(\pi)/\pi$.
   \item $\Gamma$ acts trivially on $\NN(T) / \pi \NN(T)$.
   \item as submodules of $\Drig(V)[1/t]$ we have the inclusion
   \[ \NN(V) \subseteq  \NN^+_{\cris}(V) = \cR^+_F \otimes_F \Dcris(V).\]
   \item the above inclusion induces an isomorphism
   \[ \NN(V) / \pi \NN(V) \cong \Dcris(V). \]
  \end{enumerate}

  As shown in the appendix of \cite{berger03}, we have $\DD(T)^{\psi = 1} \subseteq \pi^{-1} \NN(T)$; and the image of $\DD(T)^{\psi = 1} $ in $\pi^{-1} \NN(V) / \NN(V)$ is contained in $\Dcris(V)^{\varphi = 1}$, with $\Gamma$ acting as $\chi^{-1}$, so if $V$ has no quotient isomorphic to $\Qp$ then $\DD(T)^{\psi = 1} \subseteq \NN(T)$.

  \begin{proposition} For $T$ as above,
    the module $\left(\varphi^* \NN(T)\right)^{\psi = 0}$ is free of rank $d$ over $\cO_F[[\Gamma]]$. More precisely, we may find an $\cO_F[[\pi]]$-basis $(n_1, \dots, n_d)$ of $\NN(T)$ such that the vectors $\Big( (1 + \pi) \varphi(n_i) : i = 1, \dots, d\Big)$ are an $\cO_F[[\Gamma]]$-basis of $\left(\varphi^* \NN(T)\right)^{\psi = 0}$.
%
  \end{proposition}

  \begin{proof}
   This is proved for $F = \Qp$ in \S 3 of \cite{leiloefflerzerbes10}, and the argument adapts to general $F$ without change.
  \end{proof}

  It follows that the image of the composite map
  \[ \left(\varphi^* \NN(T)\right)^{\psi = 0} \into (\cR^+_F \otimes_F \Dcris(V))^{\psi = 0} \cong \cH_F(\Gamma) \otimes \Dcris(V) \]
  is contained in the $\Lambda_{\cO_F}$-submodule of the right-hand side generated by the finitely many elements $(1 + \pi) \varphi(n_i)$. Composing with $1 - \varphi$, we see that the image of $\NN(T)^{\psi = 1}$ is also contained in this submodule.

 \subsection{Coleman maps}

  Let us fix an $\cO_F[[\Gamma]]$-basis $(\nu_1, \dots, \nu_d)$ of $\left(\varphi^* \NN(T)\right)^{\psi = 0}$. This gives us a \emph{Coleman map}
  \[ \NN(T)^{\psi = 1} \xrightarrow{(1 - \varphi)} \left(\varphi^* \NN(T)\right)^{\psi = 0} \xrightarrow{ \operatorname{Col} } \cO_F[[\Gamma]]^{\oplus d}. \]
  We want to compare this with the map
  \[ \cL_{\cD} : \NN^+_{\cris}(V)^{\psi = 1} \xrightarrow{(1 - \varphi)} \cH_F(\Gamma) \otimes_F \Dcris(V) \cong \cH(\Gamma)^{\oplus d}, \]
  where the last isomorphism comes from a choice of $F$-basis of $\Dcris(V)$. Exactly as in \cite{leiloefflerzerbes10, leiloefflerzerbes11} we have the following:

  \begin{proposition}
   The two maps above are related by the composition of the inclusion $\cO_F[[\Gamma]]^{\oplus d} \subset \cH_F(\Gamma)^{\oplus d}$ and multiplication by a ``matrix of logarithms'' $M_{\log} \in \operatorname{Mat}_{d \times d}(\cH_F(\Gamma))$.
  \end{proposition}

  We now consider evaluation at a character $\kappa$.

  \begin{proposition}
   The image of $\NN(T)^{\psi = 1} / (\gamma - \kappa(\gamma))$ under $\cL_{\cD}$ is contained in a finitely-generated $\cO_F$-submodule of $\Dcris(V)$.
  \end{proposition}

  \begin{proof}
   This is clear since the composite of $\operatorname{Col}$ and evaluation at $\kappa$ is contained in $\cO_F^{\oplus d}$, and the image of this submodule under $\kappa(M_{\log}) \in \operatorname{Mat}_{d \times d}(F)$ is still finitely-generated.
  \end{proof}

  \begin{proposition}\label{prop:cLcontinuous}
   Let $\kappa$ be a continuous character of $\Gamma$, and if $\kappa = \chi^{-1}$ then assume that $V$ has no quotient isomorphic to the trivial representation. Then the map
   \[ \cL_{\cD, \kappa} :  H^1(F, V(\kappa^{-1})) \longrightarrow \Dcris(V) \]
   is continuous, where $H^1(F, V(\kappa^{-1}))$ has the topology from \cref{prop:banachH1} and $\Dcris(V)$ the canonical $F$-Banach topology.
  \end{proposition}

  \begin{proof}
   We consider the image under $\cL_{\cD, \kappa}$ of the unit ball in $H^1(F, V(\kappa^{-1}))$. This unit ball is precisely the image modulo torsion of $\DD(T)^{\psi = 1} / (\gamma - \kappa(\gamma))$. By assumption we have either $\kappa \ne \chi^{-1}$ or $\Dcris(V)^{\varphi = 1} = 0$, so up to scaling by a non-zero constant, this ball coincides with the image of $\NN(T)^{\psi = 1} / (\gamma - \kappa(\gamma))$. Hence its image under $\cL_{\cD, \kappa}$ is contained in a finitely-generated $\cO_F$-submodule, and thus is bounded for any $F$-Banach norm on $\Dcris(V)$.
  \end{proof}

  \begin{proposition}
   If $\kappa$ satisfies the conditions of \cref{cor:goodkappa}, then the map
   \[ \cL_{\cD, \kappa} : H^1(F, V(\kappa^{-1})) \longrightarrow \Dcris(V) \]
   is a topological isomorphism, where $H^1(F, V(\kappa^{-1}))$ has the topology from \cref{prop:banachH1} and $\Dcris(V)$ the canonical $F$-Banach topology.
  \end{proposition}

  \begin{proof}
   Under these conditions, $\cL_{\cD, \kappa}$ is both continuous and bijective. By the Banach open mapping theorem, a continuous and bijective $\Qp$-linear map between $\Qp$-Banach spaces must be a topological isomorphism.
  \end{proof}

  This shows that when $\cD = \DD^\dag_\rig(V)$, the regulator topology on $H^1(F, \cD(\kappa^{-1}))$ actually coincides with the Banach-space topology on $H^1(F, V)$ from \cref{prop:banachH1}. This proves the following special case of the ``Pessimistic Conjecture'':

  \begin{corollary}\label{pottharstclosed}
   If $\cD^+ \subseteq \cD = \DD^\dag_\rig(V)$ is a saturated $(\varphi, \Gamma)$-submodule, and $\kappa$ is a character satisfying the conditions of \cref{cor:goodkappa}, then the image of $H^1(F, \cD^+(\kappa^{-1}) \to H^1(F, \cD(\kappa^{-1})) \cong H^1(F, V)$ is closed (with respect to the Banach-space topology of $H^1(F, V)$ defined in \cref{prop:banachH1}). \qed
  \end{corollary}

 \subsection{Bloch--Kato subspaces}

  Recall that the Bloch--Kato local conditions $H^1_{\mathrm{e}}(F, V)$, $H^1_{\f}(F, V)$ and $H^1_{\mathrm{g}}(F, V)$ are the kernels of the natural maps to $H^1(F, V \otimes \BB_{\cris}^{\varphi = 1})$, $H^1(F, V \otimes \BB_{\cris})$, and $H^1(F, V \otimes \BB_{\dR})$ respectively. The groups $H^1_{\mathrm{e}}$ and $H^1_{\f}$ can be studied via the well-known exact sequence
  \[ 0 \to H^0(F, V) \to \Dcris(V) \to \Dcris(V) \oplus t(V) \to H^1(F, V), \]
  where $t(V) = \DD_{\dR}(V) / \Fil^0$, and the middle map is $x \mapsto \left((1 - \varphi)x, x \bmod \Fil^0\right)$. By construction, the image of the boundary map is $H^1_{\f}$, and the image of its restriction to $t(V)$ is $H^1_{\mathrm{e}}$.

  The above paragraph is valid for any discretely-valued closed subfield of $\CC_p$ and any representation $V$; but in the present situation, with $F$ a finite unramified extension of $\Ki$ and $V$ crystalline, we have a simplification since $(1 - \varphi) \Dcris(V) = \Dcris(V)$. It follows that $H^1_{\mathrm{e}}(F, V) = H^1_\f(F, V)$, and we have an isomorphism (the Bloch--Kato exponential)
  \[
   \exp_{F, V} : \frac{\Dcris(V)}{\Dcris(V)^{\varphi = 1} + \Fil^0 \Dcris(V)} \xrightarrow{\ \cong\ } H^1_{\f}(F, V).
  \]
  The subspace $H^1_{\mathrm{g}}$ has a somewhat more direct description as the kernel of the Bloch--Kato dual exponential map
  \[ \exp^*_{F, V} : H^1(F, V) \to \Fil^0 \Dcris(V).\]

  \begin{remark}
   Note that when $F$ is a finite extension of $\Qp$, it is conventional to define $\exp^*_{F, V}$ as the transpose of $\exp_{F, V^*(1)}$ (hence the name ``dual exponential''), and what we are calling $\exp^*_{F, V}$ would accordingly be called $\exp^*_{F, V^*(1)}$. This definition is not available in our case, since local Tate duality does not hold over infinite unramified extensions. However, a much more direct definition, not requiring the residue field to be finite, is given in \cite{kato93}, using the fact that cup-product with $\log \chi \in H^1(F, \Qp)$ gives a canonical isomorphism between $H^0(F, V \otimes \BB_{\dR}) = \DdR(V)$ and $H^1(F, V \otimes \BB_{\dR})$.
  \end{remark}

  We now express these in terms of $(\varphi, \Gamma)$-modules, assuming $V$ crystalline (but not necessarily with Hodge--Tate weights $\ge 0$). Note that there is a canonical map $\cR_F^+ \to F[[t]]$, mapping $\pi$ to $\exp(t) - 1$.

  \begin{definition}
   For $y \in \NN^+_{\cris}(V)[1/t]$, let $\partial_V(y) \in \Dcris(V)$ denote the constant term of $y$, as a power series in $F((t)) \otimes_F \Dcris(V)$.
  \end{definition}

  We shall apply this with $y \in \Drig(V)^{\psi = 1}$; then $\partial_V(y)$ factors through the image of $y$ in $H^1(F, V)$.

  \begin{proposition}
   The map
   \[ \partial_V : H^1(F, V) \to \Dcris(V) \]
   is continuous.
  \end{proposition}

  \begin{proof}
   It suffices to show that the image under this map of $H^1(F, T)$, for any lattice $T \subset V$, is contained in a finitely-generated $\cO_F$-submodule of $\Dcris(V)$. But $H^1(F, V)$ is the image of $\DD(T)^{\psi = 1}$, which is contained in $\pi^{a-1}\NN(T)$; and $\pi \NN(T)$ maps to zero. Hence $\partial_V$ factors through a finite-rank quotient $\pi^{a-1} \NN(T) / \pi \NN(T)$.
  \end{proof}

  \begin{proposition}
   The kernel of $\partial_V : H^1(F, V) \to \Dcris(V)$ is $H^1_{\f}(F, V)$.
  \end{proposition}

  \begin{proof}
   Let $y \in H^1(F, V) = \Drig(V)^{\psi = 1} / (\gamma - 1)$. Let $h \ge 1$ be such that $\Fil^{-h} \Dcris(V) = \Dcris(V)$ (so all Hodge--Tate weights of $V$ are $\le h$).\medskip

   \emph{Claim}: $\partial_V(y) = 0$ if and only if there exists a $\tilde{y} \in \Drig(V)^{\psi = 1}$ representing $y$ and such that $\tilde{y} \in \nabla (\nabla - 1) \dots (\nabla - h + 1)\cdot  \NN^+_{\cris}(V)^{\psi = 1}$.

   \emph{Proof of claim}: Since the distribution $\frac{\nabla (\nabla - 1) \dots (\nabla - h + 1)}{\gamma - 1}$ is coprime to $\gamma - 1$, it suffices to show that $\partial_V(y) = 0$ iff there is a representative $\tilde{y}$ lying in $(\gamma - 1) \cdot \NN^+_{\cris}(V)^{\psi = 1}$. Clearly if such a $\tilde{y}$ exists then $\partial_V(y) = 0$.

   Conversely, if $\partial_V(y) = 0$, then we can find a representative $\tilde{y} \in \left(\pi \NN^+_{\cris}(V)\right)^{\psi = 1}$. Thus $(1 - \varphi)(\tilde{y}) \in \cH(\Gamma) \otimes \Dcris(V)$ vanishes at the trivial character; using \cref{prop:oneminusphisurj} we obtain an element $z \in \NN^+_{\cris}(V)^{\psi = 1}$ such that $(1 - \varphi)(\tilde{y}) = (\gamma - 1)(1 - \varphi) z$. Thus $\tilde{y} - (\gamma - 1)z \in \NN^+_{\cris}(V)^{\varphi = 1} \cap \pi \NN^+_{\cris}(V) = \bigoplus_{k \ge 1} t^k \otimes \Dcris(V)^{\varphi = p^{-k}}$ (with the term for $k = 0$ excluded). Since $\gamma - 1$ is clearly surjective on this module, it follows that $\tilde{y}$ is divisible by $\gamma - 1$ as required.
   \medskip

   With this in hand, we may complete the proof. If $\partial_V(y) = 0$, then \cite[Theorem II.2]{berger03} gives us an explicit element of $\DdR(V)$ whose image under $\exp_{F, V}$ is $y$. Conversely, if $y \in H^1_{\f}(F, V)$, then we can write $y = \exp_{F, V}( (1 - p^{-1}\varphi^{-1}) w)$ for some $w \in \Dcris(V)$, and we can choose $\tilde{w} \in \NN_{\cris}^+(V)^{\psi = 1}$ with $\partial_V(\tilde w) = w$; hence it follows from Berger's formula that $\partial_V(y) = 0$.
  \end{proof}

  \begin{theorem}[Cherbonnier--Colmez, Berger]
   For $y \in H^1(F, V)$, we have
   \[ \exp^*_{F, V}(y) = (1 - p^{-1}\varphi^{-1}) \partial_V(y). \]
  \end{theorem}

  \begin{proof}
   See Theorem II.6 of \cite{berger03}.
  \end{proof}

  \begin{corollary}\label{h1fclosed}
   The spaces $H^1_\f(F, V)$ and $H^1_\mathrm{g}(F, V)$ are closed in $H^1(F, V)$.
  \end{corollary}

  \begin{proof}
   We have $H^1_\mathrm{g}(F, V) = \{ y \in H^1(F, V) : \partial_V(y) \in \Dcris(V)^{\varphi = p^{-1}} \}$, and $H^1_\f(F, V) = \{ y : \partial_V(y) = 0\}$. Since $\partial_V$ is continuous, and $\Dcris(V)^{\varphi = p^{-1}}$ is a finite-dimensional subspace of a $\Qp$-Banach space and thus closed, the result follows.
  \end{proof}

  \begin{remark}
   The assertion for $H^1_{\mathrm{g}}$ is a special case of a more general theorem of Ponsinet \cite[Prop 4.1.2]{ponsinet} showing that $H^1_{\mathrm{g}}(L, V)$ is closed in $H^1(L, V)$ for any algebraic extension $L / \Qp$ and any de Rham representation $V$ of $G_L$. However, we believe the result for $H^1_{\f}$ is new.
  \end{remark}

 \subsection{Descent to $\Qp$}

  We now suppose that $V$ is a representation of $G_{\Qp}$, so that $H^1(F, V)$ has an action of the (pro-cyclic) abelian group $U = \Gal(F / \Qp)$. Then the subspaces $H^1_{?}(F, V)$ for $? \in \{\mathrm{e}, \mathrm{f}, \mathrm{g}\}$ are clearly $U$-invariant, as is the image of $H^1(F, \cD^+)$ if $\cD^+$ is a sub-$(\varphi,\Gamma)$-module of $\DD^{\dag}_{\rig}(\Qp, V)$.

  \begin{proposition}\
   \begin{enumerate}[(i)]
    \item We have $(\res_{\Qp}^F)^{-1}\left(H^1_{\f}(F, V)\right)= H^1_\f(\Qp, V)$, and similarly for $H^1_{\mathrm{g}}$. The same holds for $H^1_{\mathrm{e}}$ if $\Dcris(V)^{\varphi = 1} = 0$ (but not otherwise).
    \item If $H^0(\Qp, \cD / \cD^+) = 0$, then \[ (\res_{\Qp}^F)^{-1}\left(\operatorname{image} H^1(F, \cD^+) \to H^1(F, V)\right) = \operatorname{image} \left(H^1(\Qp, \cD^+) \to H^1(\Qp, V)\right).\]
   \end{enumerate}
  \end{proposition}

  \begin{proof}
   Let $x \in H^1(\Qp, V)$; and write $\BB_?$ for $\BB_{\mathrm{cris}}$ or $\BB_{\mathrm{dR}}$ depending if $? = \mathrm{f}, \mathrm{g}$. Then $\res_{\Qp}^F(x) \in H^1_{?}(F, V)$ iff the image of $x$ in $H^1(\Qp, \BB_? \otimes V)$ is in the kernel of restriction to $H^1(F, \BB_? \otimes V)$. This kernel is exactly $H^1(U, \DD_?(F, V))$, and since $F$ is unramified, we have $\DD_?(F, V) = F \otimes_{\Qp} \DD_?(\Qp, V)$. So the kernel is given by $H^1(U, F) \otimes_{\Qp} \DD_?(\Qp, V)$; but $H^1(U, F) = F / (\varphi - 1) F = 0$.

   This does not work for $H^1_{\mathrm{e}}$, since the obstruction in this case lies in $H^1(U, (F \otimes_{\Qp} \Dcris(\Qp, V))^{\varphi = 1})$; and one can check that this is isomorphic to $\Dcris(\Qp,V) / (\varphi - 1) \Dcris(\Qp, V)$, which is zero iff $\Dcris(V)^{\varphi = 1}$ is. (Concretely, the unique non-trivial unramified extension of the trivial representation by itself corresponds to a class in $H^1(\Qp, \Qp)$ which is not in $H^1_{\mathrm{e}}$, since the associated extension of $\varphi$-modules does not split, but whose restriction to $\cK_\infty$ is zero, which is certainly in $H^1_{\mathrm{e}}(\cK_\infty, \Qp)$.)

   For part (ii), the obstruction is $H^1(U, H^0(F, \cD/ \cD^+))$. Since $H^0(F, \cD/ \cD^+)$ is finite-dimensional over $L$, and $U$ is topologically cyclic, its $U$-invariants and $U$-coinvariants have the same dimension; hence $H^1(U, H^0(F, \cD/ \cD^+))$ vanishes iff $H^0(U, H^0(F, \cD/ \cD^+)) = H^0(\Qp, \cD/\cD^+)$ does.
  \end{proof}

\part{Ultra-Kolyvagin systems}

\section{Ultraprimes and patched cohomology}

 \subsection{Ultrafilters}

  Recall that a \emph{ultrafilter} on a set $X$ is a family of subsets $\cU \subset \mathbb{P}(X)$ such that
  \begin{itemize}
  \item $\cU$ is an \emph{up-set}, i.e. if $A \in \cU$ then any $B \supseteq A$ is also in $\cU$ (see \ref{lem:ultravalue});
  \item if $A, B \in \cU$ then $A \cap B \in \cU$;
  \item for any $A \subset X$, precisely one of $A$ and $X \setminus A$ is in $\cU$.
  \end{itemize}

  \begin{notation}
   If $\cU$ is an ultrafilter on $X$, and $P$ is some predicate on $X$, we say ``$P$ holds for $\cU$-many $x$'' if the set $\{x \in X : P(x)\}$ is in $\cU$.
  \end{notation}

  \begin{remark}\label{lem:ultravalue}
   These conditions imply, more generally, that for any finite partition $X = A_1 \sqcup \dots \sqcup A_n$, exactly one of the $A_i$ is in $\cU$. Equivalently, for any finite set $B$ and function $f : X \to B$, there is a unique $b \in B$ such that $f(x) = b$ for $\cU$-many $x$.
  \end{remark}

  Elementary examples are the principal ultrafilters $\cU_x$, for $x \in X$, defined by $\cU_x = \{ A : x\in A\}$; the axiom of choice implies that every infinite set admits non-principal ultrafilters.

  For a family of sets (or groups, etc) indexed by $X$, say $\cM=\{M_x\}_{x\in X}$ we can then define the \emph{ultraproduct} $\cU(\cM)$, defined as the quotient of $\prod_x M_x$ by the relation $a\sim b$ if $a_x = b_x$ for $\cU$-many $x$.


 \subsection{Ultraprimes}

  Henceforth we fix a non-principal ultrafilter $\cU$ on $\mathbf{N}$.

  \begin{definition}
   Let $P$ be the set of primes of a number field $L$, and let $\cP_L=\{P\}_{n\in \NN}$. An ultraprime is a an element of the ultraproduct $\cU(\cP_L)$.

   A an ultraprime $q$ can hence be represented by a sequence of primes $(q^{(1)},q^{(2)},\dots)$, and the sequences $(q^{(n)})_{n \ge 1}$ and $(q^{\prime (n)})_{n \ge 1}$ define the same ultraprime if $q^{(n)} = q^{\prime (n)}$ for $\cU$-many $n$.

   An ultraprime is \emph{constant} if $\cU$-many of its terms are the same prime $p$.
  \end{definition}

  Given an ultraprime $q \in \cU(\cP_L)$, and an algebraic extension $E / L$, we say $q$ is \emph{unramified in $E$} if, for every finite subextension $L \subseteq_{\f} F \subseteq E$, the prime $q^{(n)}$ is unramified in $F$ for $\cU$-many $n$. This turns out to be automatic if $q$ is non-constant.

  If $q$ is unramified and $E$ is Galois, one can define an (arithmetic) Frobenius $\Frob_q \in \Gal(E / L)$, well-defined up to conjugacy, by the following rule: for each finite Galois subextension $L \subseteq_{\f} F \subseteq E$, the image of $\Frob_q$ in $\Gal(F / L)$ is the unique conjugacy class containing $\Frob_{q^{(n)}}$ for $\cU$-many $n$.

  \begin{lemma}[Ultra-Chebotarev density]
   For any $\tau \in \Gal(\overline{L}/ L)$, there exist infinitely many (non-constant) ultraprimes $q$ with $\Frob_q = [\tau]$.
  \end{lemma}

  \begin{proof}
   We first show that one such $q$ exists. Choose an $\NN$-indexed sequence of finite Galois subextensions $L = F_1 \subset F_2 \subset F_3 \dots$ whose union is $\overline{L}$; and for each $n$, choose $q^{(n)}$ to be a prime unramified in $F_n$ such that $\Frob_{q^{(n)}}$ is conjugate to $\tau$ in $\Gal(F_n / L)$. We let $q$ be the equivalence class of $(q^{(1)}, q^{(2)}, \dots)$.

   Since $\cU$ must contain all cofinite sets, and $\Frob_{q^{(n)}}$ maps to $[\tau]$ in $\Gal(F_m / L)$ for all $n \ge m$, the image of $\Frob_q$ in $\Gal(F_m / L)$ is $[\tau]$ for all $m$. As the $F_m$ exhaust $E$ the result follows.

   If we are given any finite set of ultraprimes $q_1, \dots, q_k$, then we may suppose $q^{(n)} \notin \{q_1^{(n)}, \dots, q_k^{(n)}\}$ for all $n$ and $j$, so $q$ is not equal to any of the $q_i$.
  \end{proof}


 \subsection{Local cohomology at an ultraprime}

  One can associate to each ultraprime $q$ a local Galois group $G_q$ (c.f. \cite[\S 2.3]{sweeting20}). For $q$ a constant ultraprime, this is the usual Galois group $\Gal(\overline{L}_q / L_q)$; for $q$ a non-constant ultraprime we define it to be the semidirect product $\widehat{\ZZ}(1) \rtimes \widehat{\ZZ}$, with $\Frob_q$ as a generator of the second factor, acting on the first factor (the ``inertia subgroup'') via the cyclotomic character.

  \begin{remark}
   Let us define an ``embedding $\overline{L} \into \overline{L}_q$'', for non-constant $q$, to mean an equivalence class of embeddings $\overline{L} \into \overline{L}_{q^{(n)}}$ agreeing at $\cU$-many $n$. Then a choice of such an embedding defines a specific choice of map $G_q \to \Gal(\overline{L} / L)$; so we can regard $\Frob_q$ as an element of the latter group (up to inertia if $q$ is a constant ultraprime, and well-defined ``on the nail'' otherwise). Varying the embeddings has the effect of conjugating $G_q$ inside $\Gal(\overline{L} / L)$.
  \end{remark}

  For any topological $G_L$-module $A$, and any $q \in \cP_L$, there is a natural action of $G_q$ on $A$ (well-defined up to conjugacy). Define the local cohomology groups by
  \[ H^i(L_q, A)=H^i(G_q, A), \]
  where on the RHS we take continuous cochain cohomology.

  \begin{lemma}
  	(c.f. \cite[Prop. 2.3.2]{sweeting20}) Let $q\in \cP_L$ be represented by a sequence $(q^{(n)})_{n\in\mathbb{N}}$. If $A$ is a finite, discrete $G_L$ module, then for $\cU$-many $n$ there are canonical isomorphisms
   \[ H^i(L_{q^{(n)}}, A) \cong H^i(L_q, A).\]
  \end{lemma}

  (More precisely, if we fix $\iota : \overline{L} \into \overline{L}_q$, then both $G_q$, and $G_{q^{(n)}}$ for $\cU$-many $n$, act on $A$; and the above isomorphisms are compatible with the maps from $H^i(L, A)$.)


 \subsection{Patched cohomology}

  Let $S$ be a finite set of ultraprimes, and let $A$ be a finite $G_L$-module unramified at the primes in $S$. Represent $S$ by a sequence of sets $(S^{(1)},S^{(2)},\dots)$.

  \begin{definition}
   Define the patched cohomology
   \[ H^i(G_{S},A)=\cU\left( \{H^i(L_{S^{(1)}},A),H^i(L_{S^{(2)}},A),\dots\}\right).\]
   For $B$ a profinite system of finite $G_L$-modules unramified outside $S$, say $B=\varprojlim B'$ with $B'$ finite, let
   \[ H^i(L_{S},B)=\varprojlim H^i(L_S,B').\]
  \end{definition}

  Let $S$ be a finite set of ultraprimes, and let $q\in \cP_L$. Let $A$ be a topological $G_L$-module unramified outside $S$. We define restriction maps as follows:
  \begin{enumerate}
  	\item if $A$ is finite , define
  	\[ \res^S_q: H^i(L_S,A)=\cU\left(\left\{ H^i(L_{S^{(n)}},A)\right\}_{n\in\mathbf{N}}\right)\rightarrow \cU\left(\left\{H^i(L_q,A)\right\}_{n\in\mathbf{N}}\right)\cong H^i(L_q,A).\]
  	\item if $A=\varprojlim A'$ is 	a profinite limit of finite modules, define
  	\[ \res^S_q: H^i(L_S,A)=\varprojlim H^i(L_S,A') \rightarrow^{\res^S_q} \varprojlim H^i(L_q,A') \cong H^i(L_q,A).\]
  \end{enumerate}

\section{Ultra Kolyvagin systems}

 \subsection{Tame Euler systems}

  Let $V$ be a finite-dimensional $E$-linear representation of $G_{\QQ}$ unramified outside a finite set of primes $\Sigma$ containing $p$, and $T$ a Galois-invariant $\cO$-lattice in $V$.

  \begin{notation}
   For $q \ne p$ a prime, denote by $\QQ(q)$ the maximal subfield of $\QQ(\zeta_m)$ of $p$-power degree over $\QQ$. For $r = q_1 \dots q_n$ a squarefree product of such primes, write $\QQ(r)$ for the composite of the $\QQ(q_i)$. We write $\Delta_r$ for the Galois group of $\QQ(r)$. Finally, let $\cN$ be the set of square-free products $q_1 \dots q_r$ of primes $q_i \notin \Sigma$.
  \end{notation}

  We also fix a subspace $\cF \subseteq H^1(\cK_\infty, V)$, which we suppose is invariant under the action of $\Gal(\cK_\infty / \Qp)$. Then, for $m \in \cN$ and $v \mid p$ a place of $\QQ(m)$, all embeddings $\iota : \QQ(m)_v \into \QQbar_p$ factor through $\cK_\infty$; since $\cF$ is Galois-invariant, the subspace $\left(\res_{\iota(\QQ(m)_v)}^{\cK_\infty}\right){}^{-1}(\cF)$ of $H^1(\QQ(m)_v, V)$ is independent of $\iota$, and we denote it by $H^1_{\cF}(\QQ(m)_v, V)$. We let $H^1_{\cF}(\QQ(m)_v, T)$ be the preimage of $H^1_{\cF}(\QQ(m)_v, V)$ in $H^1(\QQ(m)_v, T)$.

  \begin{definition} \label{def:tameES}
   A \emph{tame Euler system for $(T, \Sigma)$ with local condition $\cF$} is a collection of classes
   \[ c_m \in H^1(\QQ(m), T) \qquad \forall m \in \cN, \]
   such that:
   \begin{itemize}
    \item if $\ell$ is prime with $\ell m \in \cN$, then the Euler system norm relation holds,
    \[ \operatorname{norm}_{\QQ(m)}^{\QQ(\ell m)}\left(c_{\ell m}\right) = P_\ell(V^*(1), \sigma_\ell^{-1}) \cdot c_m\]
    where $\sigma_\ell$ is the arithmetic Frobenius at $\ell$, and $P_\ell(V^*(1), X) = \det\left(1 - X \sigma_\ell^{-1} \middle| V^*(1)\right)$.
    \item for all $m \in \cN$ and all finite places $v$ of $\QQ(m)$, we have
    \[ \loc_v(c_m) \in \begin{cases}
     H^1_{\f}(\QQ(m)_v, T)  & \text{for $v \nmid p$} \\
     H^1_{\cF}(\QQ(m)_v, T) & \text{for $v \mid p$}.
     \end{cases} \]
   \end{itemize}
  \end{definition}

  \begin{remark}
   Recall that if the prime $\ell$ below $v$ is not in $\Sigma$, the ramification condition is equivalent to $c_m$ being unramified above $\ell$; but it is slightly weaker if $V$ is ramified at $\ell$, since $H^1_{\f}$ is equal to the saturation of the unramified cohomology $H^1_{\mathrm{ur}}$ (which may not be saturated in general). See \cite[Lemma 1.3.5(ii)]{rubin00} for a more precise description of the quotient $H^1_\f / H^1_{\mathrm{ur}}$.
  \end{remark}

 \subsection{Hypotheses on $(T, \Sigma, \cF)$}

  We assume that the following conditions hold:
  \begin{itemize}
   \item (Hyp1) the following two ``large image'' conditions hold:
    \begin{itemize}
    \item (Hyp1a) There exists a $\tau\in G_{\QQ(\mu_{p^\infty})}$ such that $\tau$ acts trivially on $\QQ(\mu_{p^\infty})$ and such that $T/(\tau-1)T$ is a free $\cO$-module of rank $1$.
    \item (Hyp1b) There exists $\gamma \in G_{\QQ(\mu_{p^\infty})}$ such that $V^{(\gamma = 1)} = 0$.
    \end{itemize}
   \item (Hyp2) $T \otimes_{\cO}\mathbf{k}$ is an irreducible $\mathbf{k}[G_{\QQ}]$-module of rank $> 1$, where $\mathbf{k}$ is the residue field of $E$.
   \item (Hyp3) There exists a field extension $L$ containing $\mu_{p^\infty}$ such that $G_L$ acts trivially on $V$ and $H^1(L/\QQ, V/T )$ and $H^1(L/\QQ, T^\vee)$ are zero.
   \item (Hyp4) $\cF$ is closed in the Banach-space topology of $H^1(\cK_\infty, V)$.
  \end{itemize}

  We fix a choice of $\tau$ and $\gamma$ as in (Hyp1).

  \begin{remark}
   Note that (Hyp1a) is the condition $\mathrm{Hyp}(K, T)$ from \cite[\S 2.2]{rubin00}; and (Hyp1b), together with the $H^1_{\f}$ assumption away from $p$ which we imposed as part of the definition, is the condition (ii')(b) in \S 9.1 of \emph{op.cit.} (as an alternative to hypothesis (ii) in Rubin's definition of Euler systems, the existence of norm-compatible bounded classes along a $\Zp$-extension, which we are not assuming). On the other hand, (Hyp3) is very close to the hypothesis (H.3) in \cite{mazurrubin04}.
  \end{remark}

  For our chosen $\tau$, and $j \ge 1$, we say a prime $\ell$ is a \emph{$j$-Kolyvagin prime} if $\ell = 1 \bmod p^j$, $\ell \notin \Sigma$, and the conjugacy class of arithmetic Frobenius $\sigma_\ell$ acts on $T / p^j T$ as a conjugate of $\tau$. We let $\cN_j$ denote the square-free products of $j$-Kolyvagin primes.

  As explained in \S 9.1 of \emph{op.cit.}, under our present hypotheses, we can define the Kolyvagin derivative classes
  \[ \kappa_{r, j} \in H^1(\QQ, T / p^j T) \]
  for each $r \in \cN_j$. Note that by construction $\kappa_{r, j}$ is unramified outside $\Sigma \cup r$, and its localization at primes dividing $r$ is determined by the ``finite-singular comparison'' property (Theorem 4.5.4 of \cite{rubin00}, using the alternative proof in \S 9.1 of \emph{op.cit.}; this is where the hypothesis (Hyp1b) is needed).

 \subsection{Kolyvagin ultraprimes}

  We define the set of Kolyvagin ultraprimes as follows:

  \begin{definition}
  	Let $q\in \cP_{\QQ}$. Then $q$ is an \emph{Kolyvagin ultraprime} if the image of $\Frob_q$ in $\Gal\big(\QQ(T, \mu_{p^\infty}) / \QQ\big)$ is conjugate to $\tau$.
  \end{definition}

  Unwrapping the definition, if $q$ is defined by a sequence $(q^{(n)})_{n\in\mathbf{N}}$, then $q$ is a Kolyvagin ultraprime if for each $j \ge 1$, there are $\cU$-many $n$ such that $q^{(n)}$ is a $j$-Kolyvagin prime. From the results above, there exist infinitely many Kolyvagin ultraprimes.

  If $r = \{q_1, \dots, q_k\}$ is a finite set of (distinct) Kolyvagin ultraprimes, represented by sequences $(q_i^{(n)})_{n \ge 1}$, and we set $r^{(n)} = \prod_i q_i^{(n)}$, then for any $j \ge 1$, the Kolyvagin class
  \[ \kappa_{[r^{(n)}, j]} \in H^1(\QQ_{\Sigma \cup \{q_1^{(n)}, \dots, q_k^{(n)}\}}, T / p^j T)\]
  is defined for $\cU$-many $n$, and gives a class
  \[ \kappa_{[r, j]} \in H^1(\QQ_{\Sigma \cup r}, T / p^j T),\]
  where the patched Selmer group is as defined above. These classes are compatible as $j$ varies (from Rubin's lemma 4.4.13 (iii)) and hence the following is well-defined:

  \begin{definition}
   We may therefore define a class $\kappa_r \in H^1(\QQ_{\Sigma \cup r}, T)$ as the inverse limit of the $\kappa_{r, j}$.
  \end{definition}

  \begin{remark}[{see \cite[Lemma 4.4.13 (ii)]{rubin00}}]
  	For $\cU$-many $n$, the restriction of $\kappa_{r^{(n)},j}$ to $H^1(\QQ_\Sigma(r^{(n)}), T/p^j T)$ is equal to the image of $D_{r^{(n)}} c_{r^{(n)}}$ in $H^1(\QQ_\Sigma(r^{(n)}), T/p^j T)$, where $D_{r^{(n)}}$ is a certain specific element of the group ring $\Zp[\Delta_{r^{(n)}}]$.
  \end{remark}


 \subsection{Local properties at $p$ of the Kolyvagin classes}

  Let $p\in \cP_{\QQ}$ be the constant ultraprime $(p,p,\dots)$, and fix an embedding $\QQbar \into \QQbar_p$. We then have a restriction map
  \[ \res^{\Sigma \cup r}_p:\, H^1(\QQ_{\Sigma \cup r}, T)\rightarrow H^1(\Qp, T), \]
  for any finite set of Kolyvagin ultraprimes $r$; so we may regard $\res^{\Sigma \cup r}_p\left(\kappa_r\right)$ as a class in $H^1(\Qp, T)$.

  Given a $\Qp$-subspace $\cF \subset H^1(\cK_\infty, V)$ invariant under $\Gal(\cK_\infty/\Qp)$, we shall define $H^1_{\cF}(\Qp, V)
 = \left(\res_{\Qp}^{\cK_\infty}\right)^{-1} (\cF)$, and $H^1_{\cF}(\Qp, T)$ the preimage of $H^1_{\cF}(\Qp, V)$ in $H^1(\Qp, T)$ (so $H^1 / H^1_{\cF}$ is $p$-torsion-free). This extends in the obvious fashion to any finite extension of $\Qp$ contained in $\cK_\infty$.

  \begin{theorem}
   \label{thm:ultrabeef}
   Suppose that $\cF$ is closed in the Banach-space topology of $H^1(\cK_\infty, V)$, and that for all $n \in \cN$ and all embeddings $\QQ(n) \into \QQbar_p$, the Euler system class $c_n$ satisfies
   \[ \loc_p(c_n) \in H^1_{\cF}(\Qp(n), V).\]
   Then for all $r$ as above we have
   \[ \loc_p(\kappa_r) \in H^1_{\cF}(\Qp, V).\]
  \end{theorem}

  \begin{proof}
   From the definition of $H^1_{\cF}(\Qp, V)$, it suffices to show that the image of $\kappa_r$ in $H^1(\Ki, V)$ lies in the subspace $H^1_{\cF}$.

   By construction, for all $j \ge 1$, there exist $\cU$-many $n$ such that $\loc_p\left(\kappa_q\right) \bmod {p^j} = \loc_p(\kappa_{q^{(n)}, j})$. Since the field $\Ki$ contains the image of $\QQ(q^{(n)})$, the restriction of $\kappa_{q^{(n)}, j}$ to $\Ki$ has a distiguished lifting to $T$, namely $\loc_p\left(D_{q^{(n)}} c_{q^{(n)}}\right)$, which lies in $\cF$ by assumption. Hence there exist elements of $\cF$ which are within $p^j$ of $\loc_p\left(\kappa_q\right)$, for all $j$. As $\cF$ is closed, the result follows.
  \end{proof}


 \subsection{Local properties in $\Sigma$}

  Exactly the same argument applies at primes $\ell \ne p$, with the important difference that if $\cK_{\ell, \infty}$ denotes the unramified $\Zp$-extension of $\Ql$, then $H^1(\cK_{\ell, \infty}, V)$ is a finite-dimensional $\Qp$-vector space (since $G_{\cK_{\ell, \infty}}$ is topologically finitely generated modulo its wild inertia subgroup, and that wild inertia subgroup is pro-$\ell$). Hence every $\Qp$-subspace of $H^1(\cK_{\ell, \infty}, V)$ is closed.

  We are principally interested in the subspace $H^1_{\f}(\cK_{\ell, \infty}, V)$ of unramified classes (which satisfies Galois descent for $\cK_{\ell, \infty} / \Ql$ by definition, since it is the kernel of restriction to the inertia group). Hence, if the Euler system classes are locally in $H^1_{\f}$ at $\ell$, so are the ultra-Kolyvagin classes.

  \begin{remark}
   This is essentially a ``repackaging'' of Corollary 4.6.5 of \cite{rubin00}.
  \end{remark}

 \subsection{Bounding Selmer groups}\label{ssec:boundingSel}

  Let $T^\vee = \Hom(T, \mu_{p^\infty})$. We want to use Kolyvagin classes to bound a Selmer group with the local condition defined as follows:
  \[ H^1_{\f\cF}(\QQ, T^\vee) =\left \{ x \in H^1(\QQ_{\Sigma} / \QQ, T^\vee) : \loc_\ell(x) \in
  \begin{cases} H^1_{\f}(\QQ_\ell, T^\vee) & (\ell \ne p) \\
   H^1_{\cF}(\Qp, T^\vee) & (\ell = p) \end{cases} \right\}\]
  Here $H^1_{\cF}(\Qp, T^\vee)$ is the orthogonal complement of $H^1_{\cF}(\Qp, T)$ under local Tate duality.

  \begin{notation}
   Let $A$ be a $\cO$-module and $x \in A$.
   \begin{itemize}
    \item $\ord_{\cO}(x, A)$ denotes the smallest $r$ such that $\varpi^r x = 0$ (or $\infty$ if no such $r$ exists);
    \item $\ind_{\cO}(x, A)$ denotes the largest $r$ such that $x \in \varpi^r A$ (or $\infty$ if this holds for all $r$).
   \end{itemize}
  \end{notation}

  \begin{lemma}\label{lem:existsgamma}
   Let $\kappa \in H^1(\QQ, T)$ and $\eta \in H^1(\QQ, T^\vee)$; and let $L$ be a Galois extension such that $G_L$ acts trivially on $T$ and $T^\vee$.

   Then there exists $\gamma \in G_L$ such that
   \begin{itemize}
   \item $\ord_{\cO}( \eta(\gamma \tau), T^\vee / (\tau - 1) T^\vee) \ge \ord_{\cO}( \eta_L, H^1(L, T^\vee) )$.
   \item $\operatorname{ind}_{\cO}\left( \kappa(\gamma \tau), T / (\tau - 1) T\right)
    \le \operatorname{ind}_{\cO}\left( \kappa_L, H^1(L, T)\right)$.
   \end{itemize}

   Here $\operatorname{ind}_\cO(x, A)$ denotes the largest $r$ such that $x \in \varpi^r A$ (or $\infty$ if there is no largest such $r$).
  \end{lemma}

  \begin{proof}
   This follows immediately from the ``finitary'' version, Lemma 5.2.1 of \cite{rubin00}, by taking limits.
  \end{proof}

  We now have the following analogue of Rubin's Lemma 5.2.3.

  \begin{lemma}
   Suppose $C$ is a finite subset of $H^1(\QQ, T^\vee)$. Then there exists a finite sequence of ultraprimes $q_1, \dots, q_k$ such that
   \begin{itemize}
   \item Each $q_i$ is a Kolyvagin ultraprime (so $\operatorname{Fr}_{q_i}$ is in the kernel of the cyclotomic character and acts on $T$ as a conjugate of $\tau$);
   \item For $1 \le i \le k$, writing $r_i = \{q_1, \dots, q_i\}$, we have
   \[ \operatorname{ind}_{\cO}\left( \left(\kappa_{r_{i-1}}\right)_{q_i}, H^1_{\f}(\QQ_{q_i}, T) \right) \le \operatorname{ind}_{\cO}\left((\kappa_{r_{i-1}})_L, H^1(L, T)\right),\]
   \item $\{ \eta \in C : (\eta)_{q_i} = 0\ \forall 1 \le i \le k\} \subseteq H^1(L/\QQ, T^\vee) = 0$.
   \end{itemize}
  \end{lemma}

  \begin{note}
   From our assumption (Hyp3), we have
   \[ \operatorname{ind}_{\cO}(\kappa_{r_i}, H^1(\QQ, T)) = \operatorname{ind}_{\cO}( (\kappa_{r_i})_L, H^1(L, T))\qquad \forall i.\]
   If $\delta_i$ is this quantity, then we have
   \[ \delta_0 \ge \delta_1 \ge \dots \ge \delta_k \ge 0.\qedhere\]
  \end{note}

  \begin{lemma}
   Let $1 \le i \le k$. Then the cokernel of the map from $H^1_{\cF}(\QQ_{\Sigma \cup \{q_1, \dots, q_i\}}, T)$ to $\bigoplus_{i = 1}^k H^1_{\mathrm{sing}}(\QQ_{q_i}, T)$ has length at most $\delta_0 - \delta_i$.
  \end{lemma}

  \begin{proof}
   This follows by induction on $i$ in the same way as in Rubin's book.
  \end{proof}

  \begin{theorem}\label{thm:firstSelbound}
   If the assumptions $\mathrm{Hyp1, Hyp2, Hyp3}$ hold, then we have
   \[ \ell_{\cO} \left(H^1_{\f\cF}(\QQ, T^\vee)\right) \le \operatorname{ind}_{\cO}(\kappa_1).\]
   In particular, if $\kappa_1 \ne 0$, then $H^1_{\f\cF}(\QQ, T^\vee)$ is finite.
  \end{theorem}

  \begin{proof}
   We apply the previous arguments taking $C$ to be an arbitrary finite subgroup of $H^1_{\f\cF}(\QQ, T^\vee)$. Then we conclude from Poitou--Tate duality that the image of $C$ in $\bigoplus_i H^1_{\f}(\QQ_{q_i}, T^\vee)$ must have length at most $\delta_0 - \delta_k$. However, this localization map is injective by assumption, so $\ell_0(C) \le \delta_0 - \delta_k \le \delta_0$. Since $H^1_{\f\cF}(\QQ, T^\vee)$ is a torsion group, it is the union of its finite subgroups, so in fact the entire group must be bounded by $\delta_0$. Hence $\ell_{\cO} \left(H^1_{\f\cF}(\QQ, T^\vee)\right) \le \delta_0 = \operatorname{ind}_{\cO}(\kappa_1)$.
  \end{proof}

 \subsection{Interpretation via Selmer complexes}\label{sect:Selcplx1}

  We briefly recall the interpretation of \cref{thm:firstSelbound} in terms of Selmer complexes for $T$ (without explicitly referring to the $p$-torsion representation $T^\vee$). The theory of Selmer complexes is due to Nekov\v{a}\'r \cite{nekovar06}; we refer to \cite[\S 11.2]{KLZ17} for a summary. We shall consider the Selmer complex $\widetilde{R\Gamma}_{\f\cF}(\QQ, T)$ in which the local condition is the Bloch--Kato ``$\f$'' condition for $\ell \ne p$, and the local condition associated to the subspace $H^1_{\cF}(\Qp, T)$ at $p$.

  Since these are \emph{simple} local conditions in the sense of \cite[Definition 11.2.3]{KLZ17}, the cohomology groups of the Selmer complex have a straightforward classical interpretation: by \cite[Proposition 11.2.9]{KLZ17}, the degree 1 cohomology $\widetilde{H}^1_{\f\cF}(\QQ, T) \coloneqq H^1\left(\widetilde{R\Gamma}_{\f\cF}(\QQ, T)\right)$ is simply the Selmer group $H^1_{\f\cF}(\QQ, T)$ above (in which $\kappa_1$ lies); the degree 2 cohomology of the Selmer complex, $\widetilde{H}^2_{\f\cF}(\QQ, T)$, is isomorphic to the Pontryagin dual of $H^1_{\f\cF}(\QQ, T^\vee)$; and the cohomology vanishes for degrees $\ne 1, 2$.

  \begin{remark}
   To define the Selmer complex and relate it to classical Selmer groups, we do not need to know $\cF$, only the subspace $H^1_{\cF}(\Qp, V)$; it is not necessary to assume that this subspace arises as the preimage of a closed subspace of $H^1(\cK_\infty, V)$. However, this is needed for the theorems below.
  \end{remark}

  Moreover, since our local conditions are saturated at all primes, the index of $\kappa_1$ is the same whether we view it as an element of $\widetilde{H}^1_{\f\cF}(\QQ, T)$ or of $H^1(\QQ, T)$. Thus we may view \cref{thm:firstSelbound} as a relation between $\widetilde{H}^2_{\f\cF}(\QQ, T)$ and the index of $\kappa_1$ in $\widetilde{H}^1_{\f\cF}(\QQ, T)$.

  We shall apply \cref{thm:firstSelbound} in the following setting:
  \begin{itemize}
  \item $H^0(\Qp, V) = 0$;
  \item $\dim_L H^1_{\cF}(\Qp, V) = 1 + \dim_L V^{c = 1}$, where $c$ is complex conjugation.
  \end{itemize}
  From these conditions and an Euler characteristic computation, one sees that $\operatorname{rk} \widetilde{H}^1_{\f\cF}(\QQ, T) - \operatorname{rk} \widetilde{H}^2_{\f\cF}(\QQ, T) = 1$; so if the $\widetilde{H}^2$ vanishes, then the $\widetilde{H}^1$ is free of rank 1.

  \begin{proposition}\label{thm:1stSelbound-cplx}
   Under these hypotheses, if $\kappa_1 \ne 0$, then $\widetilde{H}^1_{\f\cF}(\QQ, T)$ is free of rank one, the quotient $\widetilde{H}^1_{\f\cF}(\QQ, T) / \langle \kappa_1 \rangle$ is finite, and we have
   \[ \operatorname{length}_{\cO} \left(\widetilde{H}^2_{\f\cF}(\QQ, T)\right) \le \operatorname{length}_{\cO}\left(\widetilde{H}^1_{\f\cF}(\QQ, T) / \langle \kappa_1 \rangle\right). \]

   In particular, if $\kappa_1 \ne 0$, then $\widetilde{H}^2_{\f\cF}(\QQ, V)$ is zero, and $\widetilde{H}^1_{\f\cF}(\QQ, V)$ is a one-dimensional $L$-vector space with $\kappa_1$ as a basis.
  \end{proposition}

 \subsection{Selmer complexes from $(\varphi, \Gamma)$-modules}\label{sect:Selphigamma}

  We now suppose the subspace $\cF \subseteq H^1(\cK_\infty, V)$ is given by the image of $H^1(\cK_\infty, \cD^+)$, for $\cD^+$ a sub-$(\varphi,\Gamma)$-module of $\cD = \DD^\dag_\rig(V)$. We shall suppose $\cD^+$ is saturated, so that $\cD^- \coloneqq \cD / \cD^+$ is also a $(\varphi, \Gamma)$-module.

  Following \cite{nekovar06} and \cite{pottharst13}, we define a Selmer complex over $\QQ$ by
  \[ \widetilde{R\Gamma}_\f(\QQ, V, \cD^+) \coloneqq
    \operatorname{MF}\left[ R\Gamma(\QQ_{\Sigma} / \QQ, V) \to R\Gamma(\Qp, \cD^-) \oplus \bigoplus_{\ell \in \Sigma - \{p\}} R\Gamma_{\mathrm{sing}}(\Ql, V)\right],
  \]
  where $R\Gamma_{\mathrm{sing}}(\Ql, V) = R\Gamma(\Ql^{\mathrm{nr}} / \Ql, H^1(I_\ell, V))[1]$. One can check that this is independent of $\Sigma$, and (using Poitou--Tate duality and the vanising of $H^0(\QQ, V^*(1))$ from (Hyp2)) its cohomology vanishes outside degrees 1 and 2.

  \begin{proposition}
   If $H^0(\Qp, \cD^-) = 0$ and $H^2(\Qp, \cD^+) = 0$, then the complex $\widetilde{R\Gamma}_\f(\QQ, V, \cD^+)$ is quasi-isomorphic to $\widetilde{R\Gamma}_{\f \cF}(\QQ, V)$ as defined in \cref{sect:Selcplx1}.
  \end{proposition}

  \begin{proof}
   Since $H^0(\Qp, \cD^-) = 0$ and $H^2(\Qp, \cD^+) = 0$, the local condition defined by $\cD^+$ coincides with the simple local condition defined by the image of $H^1(\Qp,\cD^+)$ in $H^1(\Qp, V)$. As we have shown above, the hypothesis $H^0(\Qp, \cD^-) = 0$ implies that this subspace is $H^1_{\mathcal{F}}(\Qp, V)$.
  \end{proof}

\part{Iwasawa theory}

\section{Analytic Euler systems over $\QQ_\infty$}\label{s:IMC}

 \subsection{Analytic Iwasawa cohomology}

  Let $E, V, T, \Sigma$ be as in the previous section. We write $\QQ_\infty \coloneqq \QQ(\mu_{p^\infty})$; and we suppose (for notational simplicity) that $V$ is \emph{odd}, i.e.~the trace of complex conjugation is $0$. (This can be worked around, at a cost of working over $\QQ(\mu_{p^\infty})^+$, rather than $\QQ(\mu_{p^\infty})$.)

  For $i \ge 0$ we set
  \[\begin{aligned}
   H^i_{\Iw, \Sigma}(\QQ_\infty, T) &\coloneqq \varprojlim_n H^i(\QQ_\Sigma / \QQ_n, T), &
   H^i_{\Iw, \Sigma}(\QQ_\infty, V) &\coloneqq H^i_{\Iw, \Sigma}(\QQ_\infty, T)[1/p].
  \end{aligned}\]
  We have the well-known isomorphism
  \[
   H^i_{\Iw, \Sigma}(\QQ_\infty, T) \cong H^i(\QQ_{\Sigma} / \QQ, T \otimes_{\cO} \cO[[\Gamma]](-\mathbf{j}))
  \]
  where $\mathbf{j}$ denotes the canonical character $G_{\QQ} \onto \Gamma \into \cO[[\Gamma]]^\times$, and similarly for $V$. These are finitely-generated $\cO[[\Gamma]]$-, resp $E[[\Gamma]]$-modules\footnote{Here we write $E[[\Gamma]]$ for $\cO[[\Gamma]][1/p]$, so $E[[\Gamma]] \subset \cH(\Gamma)$.}. There are also local versions over $\QQ_{\ell,\infty} \coloneqq \QQ_\infty \otimes \Ql$ (both for $\ell = p$ and $\ell \ne p$).

  \begin{remark}
   Note that $H^0_{\Iw, \Sigma}(\QQ_\infty, -)$ and $H^0_{\Iw}(\QQ_{\ell, \infty}, -)$ are always zero; and $H^1_{\Iw, \Sigma}$ is canonically independent of $\Sigma$ (so we shall frequently drop $\Sigma$ from the notation). However, $H^2_{\Iw, \Sigma}$ does depend on $\Sigma$.
  \end{remark}

  We now define the ``analytic'' variants following Pottharst \cite{pottharst13}:
  \begin{align*}
   H^i_{\ans}(\QQ_\infty, V) \coloneqq H^i(\QQ_\Sigma / \QQ, V \otimes \cH(\Gamma)(-\mathbf{j}))\\
   \cong \cH(\Gamma) \otimes_{E[[\Gamma]]} H^i_{\Iw, \Sigma}(\QQ_\infty, V).
  \end{align*}

  We define similarly local versions $H^i_{\an}(\QQ_{\ell,\infty}, V)$; and in the $\ell = p$ case, we can also define analytic Iwasawa cohomology for $(\varphi, \Gamma)$-modules over the Robba ring, defined as the cohomology groups of the complex
  \[ [0 \xrightarrow{\phantom{\psi - 1}} \cD \xrightarrow{\psi - 1} \cD].
  \]
  If $\cD\coloneqq \DD^\dag_\rig(V)$ then we have isomorphisms $H^i_{\an}(\QQ_{p, \infty}, \cD) \cong H^i_{\an}(\QQ_{p, \infty}, V)$, compatibly with the natural maps from both sides to $H^1(\Qp, V(\kappa^{-1}))$ for each character $\kappa$ of $\Gamma$. The structure of these cohomology groups is rather well-understood:

  \begin{proposition}
   Let $\cD$ be a $(\varphi, \Gamma)$-module over $\cR$, of rank $d$. Then
   \begin{itemize}
   \item The torsion submodule $H^1_{\mathrm{an}}(\QQ_{p, \infty}, \cD)_{\cH(\Gamma)-\mathrm{tors}}$ is isomorphic to $\cD^{\varphi = 1}$, and the quotient is free of rank $d$ over $\cH(\Gamma)$.

   \item $H^2_{\mathrm{an}}(\QQ_{p, \infty}, \cD)$ is a finitely-generated $E$-vector space, isomorphic as a $\cH(\Gamma)$-module to $(\cD^\vee(1)^{\varphi = 1})^*$ (where $\vee$ denotes $(\varphi, \Gamma)$-module dual, and $*$ denotes $E$-linear dual).
   \end{itemize}
  \end{proposition}

  \subsubsection*{Seminorms}The ring $\cH(\Gamma)$ is isomorphic to the ring of power series convergent on the open unit disc, and we can write it as $\bigcap_{r < 1} \cH_r$, where $\cH_r$ is the algebra of power series convergent on the closed disc of radius $r$; this is a Banach algebra, with open ball $\cH_r^\circ$. This gives a natural seminorm $\|-\|_r$ on $H^i_{\ans}(\QQ_\infty, V)$, for which the unit ball is the image of $H^i(\QQ_\Sigma / \QQ, T \otimes \cH_r^\circ)$ (and similarly for the local spaces). Note we don't claim that this is a norm, only a seminorm.

 \subsection{Selmer complexes over $\QQ_\infty$}

  If $\cD^+ \subset \cD = \DD^{\dag}_{\rig}(V)$ is a $(\varphi, \Gamma)$-submodule, then we can form the associated Selmer complex
  \[ \widetilde{R\Gamma}_{\an}(\QQ_\infty, V, \cD^+) \coloneqq
  \operatorname{MF}\left[ R\Gamma_{\ans}(\QQ_\infty, V) \to R\Gamma_{\an}(\QQ_{p, \infty}, \cD/\cD^+) \oplus \bigoplus_{\ell \in \Sigma - \{p\}} R\Gamma_{\an, \mathrm{sing}}(\QQ_{\ell, \infty}, V)\right] \]
  which is a perfect complex of $\cH(\Gamma)$-modules, independent of the choice of $\Sigma$. Here $R\Gamma_{\an, \mathrm{sing}}(\QQ_\infty \otimes \Ql, V)$ is the degree $\ge 2$ truncation of $R\Gamma_{\an}(\QQ_\infty \otimes \Ql, V)$.

  The cohomology of $\widetilde{R\Gamma}_{\an}(\QQ_\infty, V, \cD^+)$ can be studied via the long exact sequence of the mapping fibre: under our running assumptions on $V$, the cohomology vanishes for $i \notin \{1, 2\}$ and
  \[
   \widetilde{H}^1_{\an}(\QQ_\infty, V, \cD^+) = \{ x \in H^1_{\an}(\QQ_\infty, V) : \loc_p(x) \in \operatorname{image} H^1_{\an}(\QQ_\infty, \cD^+)\}.
  \]

  \begin{remark}
   The ranks of the remaining two groups are related by the Euler characteristic formula
   \[ \operatorname{rk} \widetilde{H}^1_{\an}(\QQ_\infty, V, \cD^+) -  \operatorname{rk} \widetilde{H}^2_{\an}(\QQ_\infty, V, \cD^+) = \operatorname{rk}\cD^+ - \tfrac{1}{2}\dim V. \]
   (If we drop the assumption that $V$ be odd, and replace $\QQ_\infty$ with any field intermediate between the cyclotomic $\Zp$-extension and $\QQ(\mu_{p^\infty})^+$, then this formula remains true with $\tfrac{1}{2}\dim V$ replaced by $\dim V^{(c = 1)}$.)
  \end{remark}

 \subsection{Analytic Euler systems}

  Recall $\cN$ is the set of square-free products of primes not in $\Sigma$. For $m \in \cN$, let $\QQ_\infty(m) = \QQ_\infty \QQ(m)$.

  \begin{definition}
   \label{def:anES}
   By an \emph{analytic Euler system} for $V$, we mean a collection of classes
   \[ c_m \in H^1_{\an}(\QQ_\infty(m), V) \qquad \forall m \in \cN \]
   satisfying the following conditions:
   \begin{itemize}
    \item The Euler system norm relation holds as $m$ varies.
    \item For any $r < 1$, there exists $C_r < \infty$ (independent of $m$) such that $\| c_m\|_r \le C_r$ for all $m \in \cN$.
   \end{itemize}
   We say $(c_m)_{m \in \cN}$  \emph{satisfies the local condition $\cD^+$ at $p$} if for all $m \in \cN$ and all primes $v \mid p$ of $\QQ(m)$, we have
   \[ \loc_v(c_m) \in \operatorname{image} \left( H^1_{\mathrm{an}}(\QQ_\infty(m)_v, \cD^+) \to H^1_{\mathrm{an}}(\QQ_\infty(m)_v, V)\right).\]
    Equivalently,
   \( \loc_v c_m \in \widetilde{H}^1_{\an}(\QQ_\infty(m), V, \cD^+). \)
  \end{definition}

  \begin{remark}
   This setting is a little more general than considered in \cite{loefflerzerbes16} where we worked with finite-order distributions, rather than arbitrary locally-analytic distributions as above; any Euler system in the sense of \emph{op.cit.} is also an analytic Euler system in the present sense (but not conversely).
  \end{remark}

  \begin{proposition}
   Let $r < 1$ and let $C_r$ be as in the definition. Assume $V$ satisfies (Hyp2). Then for any $\kappa : \Gamma \to E$ which factors through $\cH_r$, the classes $C_r \cdot \kappa(c_m) \in H^1(\QQ(m), V(\kappa^{-1}))$, for $m \in \cN$, are the images in $V$ of a uniquely determined tame Euler system for $(T(\kappa^{-1}), \Sigma)$ with local condition $\cF$, in the sense of \cref{def:tameES}, where $\cF$ is the image of $H^1(\cK_\infty, \cD^+(\kappa^{-1}))$ in $H^1(\cK_\infty, V(\kappa^{-1}))$.
  \end{proposition}

  \begin{proof}
   The assumption (Hyp2) implies that $H^1(\QQ(m), T)$ is torsion-free for all $m$ and hence can be identified with a lattice in $H^1(\QQ(m), V)$. Hence the classes $C_r \cdot \kappa(c_m) \in H^1(\QQ(m), T(\kappa^{-1}))$ are well-defined, and satisfy the Euler-system norm relations exactly (not just modulo torsion).

   Moreover, for each prime $v \nmid p$, the image of $H^1_{\an}(\QQ_\infty(m)_v, V)$ in $H^1(\QQ(m)_v, V(\kappa^{-1}))$ is contained in $H^1_{\mathrm{f}}(\QQ(m)_v, V(\kappa^{-1}))$; this is well known for Iwasawa cohomology, and the result for analytic Iwasawa cohomology follows by extending scalars to $\cH(\Gamma)$. The local condition for $v \mid p$ clearly holds by construction, so the classes $C_r \cdot \kappa(c_m)$ form a tame Euler system for $T(\kappa^{-1})$.
  \end{proof}

  Our goal will be to show that the existence of an analytic Euler system with local condition $\cD^+$ gives bounds for $\widetilde{H}^2_{\an}(\QQ_\infty, V, \cD^+)$. To do this, we will need to study descent properties of the Pottharst Selmer complexes and their relation to classical $p$-torsion Selmer groups over $\QQ$.

 \subsection{Descent of the Selmer complex and rank bounds}

  Essentially by construction we have the base-change property
  \[ \widetilde{R\Gamma}_\f(\QQ, V, \cD^+) = E \otimes^{\mathbf{L}}_{\cH(\Gamma)} \widetilde{R\Gamma}_{\an}(\QQ_\infty, V, \cD^+),\]
  which translates into short exact sequences for each $i$
  \[ 0 \to \widetilde{H}^i_{\an}(\QQ_\infty, V, \cD^+)_{\Gamma} \to \widetilde{H}^i_{\f}(\QQ, V, \cD^+) \to
  \widetilde{H}^{(i + 1)}_{\an}(\QQ_\infty, V, \cD^+)^{\Gamma} \to 0, \]
  and likewise for twists by characters of $\Gamma$. Moreover, for almost all $\kappa$ the module
  \[ \widetilde{H}^{(i + 1)}_{\an}(\QQ_\infty, V(\kappa^{-1}), \cD^+(\kappa^{-1})) = \widetilde{H}^{(i + 1)}_{\an}(\QQ_\infty, V, \cD^+)[\Gamma = \kappa]\]
  is zero.

  \begin{theorem}\label{thm:H2tors}
   Suppose (Hyp1)--(Hyp3) hold, that $V$ is crystalline at $p$, and $\mathbf{c}$ is an analytic Euler system satisfying the local condition $\cD^+$, such that $e_{\eta} \cdot c_1 \ne 0$ for all characters $\eta$ of $\Gamma_{\mathrm{tors}}$ (where $e_\eta \in \cH(\Gamma)$ is the idempotent attached to $\eta$). Then $\widetilde{H}^2_{\an}(\QQ_\infty, V, \cD^+)$ is $\cH(\Gamma)$-torsion; equivalently, the rank of $\widetilde{H}^1_{\an}(\QQ_\infty, V, \cD^+)$ is $\operatorname{rk}(\cD^+) - \tfrac{1}{2}\dim V$.
  \end{theorem}

  \begin{proof}
   From the descent exact sequence for $i = 0$ we see that $\widetilde{H}^1_{\an}(\QQ_\infty, V, \cD^+)$ is torsion-free. Hence, if $c_1 \ne 0$, then the set of characters $\kappa$ such that $\kappa(c_1) = 0$ is a discrete subset of the open disc. Thus, for almost all $\kappa$, the collection $\left(C \cdot \kappa(c_m)\right)$ is a tame Euler system for $T(\kappa^{-1})$, for a suitable $C$ (and $T(\kappa^{-1})$ also satisfies Hyp1--Hyp3). Moreover, at primes $\ell \in \Sigma - \{p\}$ these classes land in $H^1_{\f}$; and at $p$ they land in the cohomology of $\cD^+$. As we have seen above, for almost all $\kappa$ this image is a closed subspace.

   Hence, using \cref{thm:ultrabeef}, for almost all $\kappa$ the $p$-torsion Selmer group of $T^\vee(\kappa)$, with the local condition at $p$ defined by the orthogonal complement of $\cD^+$, is finite; thus $H^1_{\f \cF^\perp}(\QQ, V^*(1 + \kappa)) = 0$ for almost all $\kappa$. From local Tate duality it is clear that $H^2(\Qp, \cD^+(\kappa))$ is also zero for almost all $\kappa$, so we conclude that $ \widetilde{H}^2_{\f}(\QQ, V(\kappa), \cD^+(\kappa))$ has the same property. From the descent exact sequence, almost all specializations of $\widetilde{H}^2_{\an}(\QQ_\infty, V, \cD^+)$ are zero, and thus it is torsion over $\cH(\Gamma)$.
%
%
%
  \end{proof}

  \begin{remark}
   In particular, nontrivial Euler systems with local condition $\cD^+$ cannot exist unless $\operatorname{rk}(\cD^+) \ge \tfrac{1}{2}\dim V + 1$. We are most interested in the case when equality holds.
  \end{remark}

 \subsection{Statement of the main theorem}

  There is a well-behaved theory of \emph{characteristic ideals} for torsion coadmissible $\cH(\Gamma)$-modules (see \cite[\S 3.4]{Nakamura-Iwasawa}). This associates to each such module $M$ a non-zero principal ideal $(f_M) = \operatorname{char}_{\cH(\Gamma)}(M)$, with the property that
  \[ \operatorname{length}_{\cH(\Gamma)_\fP}(M_\fP) = \ord_\fP(f_M) \]
  for every maximal ideal $\fP$. This is compatible via base-extension with the usual theory of characteristic ideals for $\cO[[\Gamma]]$ or $E[[\Gamma]]$-modules.

  Our goal in the next few sections is to prove the following theorem:

  \begin{theorem}\label{thm:mainSelbound}
   Suppose $V, \cD, \mathbf{c}$ satisfy the hypotheses of \cref{thm:H2tors}, and moreover that
     $\operatorname{rk}(\cD^+) = \tfrac{1}{2} \dim(V) + 1$.
   Then we have the divisibility
   \[
    \operatorname{char}_{\cH(\Gamma)} \widetilde{H}^2(\QQ_\infty, V, \cD^+) \mathrel{\Big|} \operatorname{char}_{\cH(\Gamma)} \left(\frac{\widetilde{H}^1_{\mathrm{an}}(\QQ_\infty, V, \cD^+)}{\cH(\Gamma) \cdot c_1}\right) \cdot \cE^1 \cdot \cE^2, \tag{$\star$}
   \]
   where $\cE^1$ and $\cE^2$ are ``error terms'' defined by
   \[ \cE^1 = \operatorname{char}_{\cH(\Gamma)}
      \ker\left( H^1_{\an}(\QQ_{p, \infty}, \cD^-)_{\mathrm{tors}} \to H^2_{\an}(\QQ_{p, \infty}, \cD^+)\right), \qquad \cE^2 =
       \operatorname{char}_{\cH(\Gamma)}
           H^2_{\an}(\QQ_{p, \infty}, \cD^+).\]
  \end{theorem}

  \begin{remark}[Remarks on the error terms] \
   \begin{itemize}
   \item We conjecture that such a divisibility should hold without the error terms $\cE^i$, but this would require different methods than the present paper.

   \item Compare e.g.~Corollary 12.3.5 of \cite{KLZ17}, which applies in the case where $\cD^+$ arises from a subrepresentation $V^+$. In \emph{op.cit.} the error term $\cE^2$ is present, and the running assumption Hypothesis 12.3.1 of \emph{op.cit.} implies that $\cE^1$ is trivial.

   \item A version of the error term $\cE^2$ also appears in Theorem 13.4 of \cite{kato04}, for example (it is the difference between the modules $\mathbf{H}^2(T)$ and $\mathbf{H}^2(T)_0$ in Kato's notation); note that in Kato's setting $\cD^- = 0$, so there is no analogue of $\cE^1$.

   \item One can check that $H^1_{\an}(\QQ_{p, \infty}, \cD^-)_{\mathrm{tors}}$ is isomorphic as a $\cH(\Gamma)$-module to $(\cD^-)^{\varphi = 1}$, a finite-dimensional $E$-vector space. Similarly $H^2_{\an}(\QQ_{p, \infty}, \cD^+)$ is isomorphic to $\left((\cD^{+\vee})^{\varphi = 1}\right)^*$ where $\vee$ denotes the $(\varphi, \Gamma)$-module dual $\Hom_{\cR}(-, \cR(1))$. Thus if $\Dcris(V)$ has no $\varphi$-eigenvalues that are powers of $p$, the error terms $\cE^i$ are both trivial.

   \end{itemize}

  \end{remark}

\section{Characteristic ideals and integral structures}

 We now embark on the proof of \cref{thm:mainSelbound}. Evidently, it suffices to prove that for each maximal ideal $\mathfrak{P}$, the order of vanishing of the left-hand side of ($\star$) at $\fP$ is at most that of the right-hand side. We want to give an interpretation of this order of vanishing in terms of \emph{finite} modules associated to a sequence of maximal ideals tending to $\fP$, adapting an approach used for $\cO[[\Gamma]]$-modules in \cite{mazurrubin04}. Since the residue field of any maximal ideal is a finite extension of $E$, and characteristic ideals are compatible with base-extension in the coefficient field, it is sufficient to consider the case when $\fP$ has residue field $E$; thus we may assume $\fP$ corresponds to a character $\kappa : \Gamma \to E^\times$.

 \subsection{Module-theoretic preliminaries}

  Since $\kappa$ is continuous, we must have $|\kappa(\gamma_1) - 1| < 1$, where $\gamma_1$ generates the subgroup $\Gamma_1 \cong (1 + p\Zp)^\times \subset \Gamma$. Since $\kappa$ is also $E$-valued, we must have $|\kappa(\gamma_1) - 1| \le |\varpi|$. Thus $\kappa$ extends to the completion $A = E\langle X\rangle$ of $e \cH(\Gamma)$, where $e$ is the idempotent determined by the restriction of $\kappa$ to $\Gamma_{\mathrm{tors}}$, and $X = \frac{[\gamma_1] - 1}{\varpi}$.

  \begin{definition}
   By a \emph{good sequence} tending to $\kappa$, we mean a sequence of characters $\kappa_n : \Gamma \to E^\times$ for $n \ge 1$, all agreeing with $\kappa$ on $\Gamma_{\tors}$, such that $\kappa_n \to \kappa$ as $n \to \infty$ and $\kappa_n \ne \kappa$ for all $n \gg 0$. We write $\fP_n$ for the maximal ideal corresponding to $\kappa_n$.
  \end{definition}

  (For instance, we can take $\kappa_n$ to be the character sending $X$ to $\lambda + \varpi^n$, where $\lambda = \frac{\kappa(\gamma_1) - 1}{\varpi}$.)

  \begin{proposition}\label{prop:Acircmodulegrowth}
   Let $A^\circ$ be the integral Tate algebra $\cO\langle X\rangle$, and $M^\circ$ a finitely-generated $A^\circ$-module. Then, for any good sequence $(\kappa_n) \to \kappa$, we have the following:
   \begin{enumerate}[(i)]
    \item For all but finitely many $n$, the submodule $M^\circ[\fP_n]$ is finite and  $M^\circ / \fP_n M^\circ$ has $\cO$-rank equal to the $A$-rank of $M$, where $M = M^\circ[1/p]$.

    \item The orders of the groups $M^\circ[\fP_n]$ are uniformly bounded as $n \to \infty$.

    \item The $p$-torsion submodule $\left(M^\circ / \fP_n M^\circ\right)_{\text{$p$-tors}}$ is finite for all $n$, and as $n \to \infty$ we have
    \[ \ell_{\cO}\left( \left(M^\circ / \fP_n M^\circ\right)_{\text{$p$-tors}}\right) = h \cdot \ord_{\cO}\left(\kappa_n(\gamma_1) - \kappa(\gamma_1)\right) + O(1), \]
    where $h \coloneqq \ord_{\fP} \left(\operatorname{char}_{A} M_{\mathrm{tors}}\right)$.
   \end{enumerate}
  \end{proposition}

  \begin{notation}
   We express the conclusion of (iii) more concisely as ``the groups $\left(M^\circ / \fP_n M^\circ\right)_{\text{$p$-tors}}$ grow at rate $h$''.
  \end{notation}

  \begin{proof}
   For part (i), we note that $M_{\mathrm{tors}}$ is a torsion module over a Dedekind domain and hence has finite support; and since $\kappa_n \to \kappa$ and $\kappa_n \ne \kappa$ for $n \gg 0$, the ideal $\fP_n$ is outside this support for all but finitely many $n$.

   For (ii): $M^\circ$ is finitely generated, and $A^\circ$ is Noetherian; hence there is a maximal finite $A^\circ$-submodule of $M^\circ$, whose order is clearly an upper bound for the $M^\circ[\fP_n]$ once $n$ is large enough that this group is finite.

   For (iii), using the fact that $A$ is a Dedekind domain, we can readily reduce to the case when $M^\circ$ is either a free module, a module annihilated by $p^k$ for some $k$, or a module of the form $A^\circ / P(X) A^\circ$ where $P$ is a monic polynomial over $\cO$. The first two cases are obvious, since $h = 0$ and the left-hand side is bounded; and for the last case, $M^\circ / \fP_n M^\circ \cong \cO / P(\lambda_n - \lambda) \cdot \cO$, where $\lambda_n = \frac{\kappa_n(\gamma_1) - 1}{\varpi}$. For $n \gg 0$ the valuation of $P(\lambda_n - \lambda)$ is precisely $h \cdot \ord_\cO (\lambda_n - \lambda)$, where $h$ is the order of vanishing of $P$ at $\lambda$; and since $P$ generates the characteristic ideal of $A / P$, this gives the result.
  \end{proof}

  \begin{remark}
   Observe that part (ii) holds whether or not $\fP$ itself is in the support of $M_{\mathrm{tors}}$; thus the $\kappa_n$-isotypic \emph{submodules} remain bounded as $\kappa_n \to \kappa$, whereas (iii) shows that the $\kappa_n$-isotypic \emph{quotients} blow up (at a precisely controlled rate). Note also that the asymptotic behaviour of $M^\circ / \fP_n M^\circ$ as $n \to \infty$ depends only on the isomorphism class of $M = M^\circ[1/p]$; concretely, if we consider a map of finitely-generated $A^\circ$-modules which is an isomorphism after inverting $p$, then the maps it induces on the $\fP_n$-isotypic quotients will have kernels and cokernels that are bounded independently of $n$.
  \end{remark}

 \subsection{A ``fake'' integral complex}

  Since there is no natural integral version of the cohomology complex
  \[ A \otimes_{\cH(\Gamma)} R\Gamma_{\an}(\QQ_{p, \infty}, \cD^+) \cong R\Gamma(\Qp, \cD^+ \otimes_E A),\]
  we shall simply choose one \emph{arbitrarily}, and show that this arbitrary choice does not matter.

  \subsubsection*{Complexes over $A^\circ$} We can represent the module $R\Gamma(\Qp, \cD^+ \otimes_E A)$ by a bounded complex of finitely-generated free $A$-modules concentrated in degrees $\{0, 1, 2\}$, by results of \cite{KPX}. We choose a $A^\circ$-lattice $R\Gamma(\Qp, \cD^+ \otimes A)^\circ$ in $R\Gamma(\Qp, \cD^+ \otimes A)$; that is, for some choice of representing complex consisting of finite free $A$-modules, we choose $A^\circ$-lattices in each term in the complex that are compatible under the differentials. Scaling the lattices appropriately we may suppose that the natural map $R\Gamma(\Qp, D^+ \otimes A)^\circ \to R\Gamma(\Qp, V \otimes A)$ factors through $R\Gamma(\Qp, T \otimes A^\circ)$.

   Note that the mapping fibre of $ R\Gamma(\Qp, \cD^+ \otimes A)^\circ \to R\Gamma(\Qp, T \otimes_{\cO} A^\circ)$ is an $A^\circ$-model of $R\Gamma(\Qp, \cD^- \otimes A)$, which we denote by $R\Gamma(\Qp, \cD^- \otimes A)^\circ$. We thus have three perfect complexes of $A^\circ$-modules fitting into a distinguished triangle
   \[ R\Gamma(\Qp, \cD^+ \otimes A)^\circ \to R\Gamma(\Qp, T \otimes A^\circ) \to R\Gamma(\Qp, \cD^- \otimes A)^\circ \to [+1], \]
   which agree after inverting $p$ with the usual cohomology complexes over $A$.

   \subsubsection*{Complexes over $\cO$}
%


   We now define complexes ``at finite level'' by (derived) tensor product with $A^\circ /\fP A^\circ \cong \cO$. After inverting $p$ these compute the cohomology of the twists $\cD^{\pm}(\kappa^{-1})$, so we denote them by $R\Gamma(\Qp, \cD^+(\kappa^{-1}))^\circ$ etc; and similarly for $\kappa_n$ in place of $\kappa$.

   \begin{proposition}
    Let $(\kappa_n) \to \kappa$ be a good sequence. If $i \notin \{1, 2\}$, the sequence of cohomology groups $H^i\left(\Qp, \cD^\pm(\kappa_n^{-1})\right)^\circ$ is finite and uniformly bounded for $n \gg 0$.

    For $i = 1$, these groups have rank $\operatorname{rk}\cD^{\pm}$ for all $n \gg 0$, and their torsion submodules grow with rate $\dim_E H^0(\QQ_{p, \infty}, \cD^{\pm})_{(\fP)}$.

    For $i = 2$, these groups are finite for $n \gg 0$ and their growth rate is $\dim H^2_{\an}(\QQ_{p, \infty}, \cD^{\pm})_{(\fP)}$.
   \end{proposition}

   \begin{proof}
    Since the ideals $\fP_n$ are principal ideals generated by a non-torsion element, it follows formally from the definitions of our complexes that for each $i$ we have exact sequences for each $i$
    \[
     0 \to H^i\left(\QQ_{p, \infty}, \cD^\pm\right)^\circ / \fP_n \to H^i
     \left(\Qp, \cD^\pm(\kappa_n^{-1})\right)^\circ \to H^{(i+1)}\left(\QQ_{p, \infty}, \cD^\pm\right)^\circ[\fP_n] \to 0.
    \]
    The result now follows by applying \cref{prop:Acircmodulegrowth} to the cohomology groups of the $A^\circ$-module complexes.
   \end{proof}

   \begin{notation}
    Let $\cL^\circ_n$ denote the image of $H^1(\Qp, \cD^+(\kappa_n^{-1}))^\circ$ in $H^1(\Qp, T(\kappa_n^{-1}))$.
   \end{notation}

   Thus the saturation of $\cL_n^\circ$ is the submodule $\cL_n^{\mathrm{sat}} = H^1_{\cF}(\Qp, T(\kappa_n^{-1}))$ (which is by definition saturated).

   \begin{proposition}
    The quotient $H^1_{\cF}(\Qp, T(\kappa_n^{-1})) / \cL^\circ_n$ is a finite $\cO$-module, whose growth rate as $n \to \infty$ is equal to the order of vanishing at $\fP$ of
    \[
     \operatorname{char}_{\cH(\Gamma)} \left(\ker H^1_{\an}(\QQ_{p, \infty}, \cD^-)
     \xrightarrow{\ \partial\ } H^2_{\an}(\QQ_{p, \infty}, \cD^+) \right)_{\mathrm{tors}}
    \]
    (i.e., the term $\cE^1$ in \cref{thm:mainSelbound}).
   \end{proposition}

   \begin{proof}
    The quotient $H^1_{\cF}(\Qp, T(\kappa_n^{-1})) / \cL^\circ_n$ is precisely the $p$-torsion subgroup of $H^1(\Qp, T(\kappa_n^{-1})) / \cL^\circ_n$. The latter is identified with the kernel of the boundary map
    \[ H^1(\Qp, \cD^-(\kappa_n^{-1}))^\circ \to H^2(\Qp, \cD^+(\kappa_n^{-1}))^\circ, \]
    and by \cref{prop:Acircmodulegrowth}, the growth of its torsion subgroup is controlled by the stated characteristic ideal.
   \end{proof}

  \subsection{Integral Selmer complexes}

   We now define global Selmer complexes over $A^\circ$, using $R\Gamma(\Qp, \cD^+ \otimes A)^\circ$ as the local-condition complex at $p$; and analogously over $\cO$ for the twists $\kappa$ or $\kappa_n$.

   By construction, $\widetilde{R\Gamma}_{\f}(\QQ_\infty, V, \cD^+)^\circ$ is an $A^\circ$-model of the Pottharst Selmer complex. Applying \cref{prop:Acircmodulegrowth} to its cohomology, we see that the order of vanishing of $\operatorname{char}_{\cH(\Gamma)} \widetilde{H}^2(\QQ_\infty, V, \cD^+)$ (the left-hand side of \cref{thm:mainSelbound}) at $\fP$  is encoded by the growth rate of the finite groups $\widetilde{H}^2_{\f}\left(\QQ, V(\kappa_n^{-1}), \cD^+(\kappa_n^{-1})\right)^\circ$ as $n \to \infty$.

   There is a natural surjective map
   \[ \widetilde{H}^2_{\f}\left(\QQ, V(\kappa_n^{-1}), \cD^+(\kappa_n^{-1})\right)^\circ \twoheadrightarrow H^2\left(\Qp, \cD^+(\kappa_n^{-1})\right)^\circ, \]
   whose kernel is the Pontryagin dual of a subgroup of the Selmer group for $T^\vee$; more precisely, it is the dual of
   \[ \tag{$\dag$} \ker\left( H^1_{\f \setminus \{p\}} (\QQ, T^\vee(\kappa_n)) \to (\cL_n^\circ)^\vee\right), \]
   where ``$\f \setminus \{p\}$'' signifies that we apply the Bloch--Kato conditions away from $p$ (and no condition at $p$).
   This group contains the Selmer group $\widetilde{H}^2_{\f\cF}(\QQ, T(\kappa_n^{-1}))$ appearing in \cref{thm:1stSelbound-cplx}, which would correspond to replacing $\cL_n^\circ$ with its saturation; and the quotient of the two injects into the finite group $\left(\cL_n^{\mathrm{sat}} / \cL_n^{\circ}\right)^\vee$.

   Thus we have
   \[ \ell_{\cO} \widetilde{H}^2_{\f}\left(\QQ, V(\kappa_n^{-1}), \cD^+(\kappa_n^{-1})\right)^\circ = \ell_{\cO} H^2\left(\Qp, \cD^+(\kappa_n^{-1})\right)^\circ + \ell_{\cO} \widetilde{H}^2_{\f\cF}(\QQ, T(\kappa_n^{-1})) + \ell_{\cO} I,\]
   where $I$ is the image of $(\dag)$ in $\left(\cL_n^{\mathrm{sat}} / \cL_n^{\circ}\right)^\vee$.

   Since $(\kappa_n)$ is a good sequence, for all but finitely many $n$ the character $\kappa_n$ satisifies the conditions of \cref{pottharstclosed}. Thus we may apply \cref{thm:1stSelbound-cplx} to give the bound
   \[ \ell_{\cO} \widetilde{H}^2_{\f\cF}(\QQ, T(\kappa_n^{-1})) \le \ell_{\cO}\left(\widetilde{H}^1_{\f\cF}(\QQ, T(\kappa_n^{-1})) / \langle c_1\rangle\right).\]

   So we must now relate $\widetilde{H}^1_{\f\cF}(\QQ, T(\kappa_n^{-1}))$ (defined using the simple Selmer complex associated to the saturated local condition associated to $\cL_n^{\mathrm{sat}}$) and $\widetilde{H}^1_{\f\cF}(\QQ, V(\kappa_n^{-1}), \cD^+(\kappa_n^{-1}))$ (defined using the complex $R\Gamma(\Qp, \cD^+(\kappa_n^{-1}))^\circ$). The map
   \[ \widetilde{H}^1_{\f\cF}(\QQ, V(\kappa_n^{-1}), \cD^+(\kappa_n^{-1})) \to H^1_{\f \setminus \{p\}}(\QQ, T(\kappa_n^{-1}))\]
   has kernel uniformly bounded in $n$; and the image of this map is the Selmer group with local condition $\cL_n^\circ$, which is contained in $\widetilde{H}^1_{\f\cF}(\QQ, T(\kappa_n^{-1}))$ with finite index. The quotient of these two is precisely the image of $H^1_{\f\cF}(\QQ, T(\kappa_n^{-1}))$ in $\cL_n^{\mathrm{sat}} / \cL_n^{\circ}$. We denote this by $J$; thus we have
   \[ \ell_{\cO}\left(\widetilde{H}^1_{\f\cF}(\QQ, T(\kappa_n^{-1})) / \langle c_1\rangle\right) \le \ell_{\cO}\left(\widetilde{H}^1(\QQ, V(\kappa_n^{-1}), \cD^+(\kappa_n^{-1}))^\circ / \langle c_1\rangle\right) + \ell_\cO J + O(1).\]

   It follows from Poitou--Tate duality that $I$ and $J$ are orthogonal complements of each other, so their lengths sum to $\ell_{\cO}(\cL_n^{\mathrm{sat}}/\cL_n^\circ)$. Putting these ingredients together we have
   \begin{multline*}
    \ell_{\cO} \widetilde{H}^2_{\f}\left(\QQ, V(\kappa_n^{-1}), \cD^+(\kappa_n^{-1})\right)^\circ \le \ell_{\cO} H^2\left(\Qp, \cD^+(\kappa_n^{-1})\right)^\circ + \ell_{\cO}(\cL_n^{\mathrm{sat}}/\cL_n^\circ) \\+ \ell_{\cO}\left(\widetilde{H}^1(\QQ, V(\kappa_n^{-1}), \cD^+(\kappa_n^{-1}))^\circ / \langle c_1\rangle\right) + O(1).
   \end{multline*}

   As $n \to \infty$, the growth rate of $\ell_{\cO} H^2\left(\Qp, \cD^+(\kappa_n^{-1})\right)^\circ$ is given by the order of vanishing of the characteristic ideal of $H^2_{\an}(\QQ_{p, \infty}, \cD^+)$ at $\fP$ (the term $\cE^2$ in the theorem); and the growth rate of $\cL_n^{\mathrm{sat}}/\cL_n^\circ$ is given by the term $\cE^1$ in the theorem. This completes the proof of \cref{thm:mainSelbound}.\qed
\part{Applications}

%

 \section{The Rankin--Selberg case}

  Let $p\geq 5$ be prime. Let $f,g$ be normalized new modular eigenforms of weights $r > r' \ge 1$ and levels $N_f,N_g$ coprime to $p$. Let $L$ be a finite extension of $\QQ$ containing the Hecke eigenvalues of $f$ and $g$.
  
  \subsection{Galois representations} For each prime $\pp \mid p$ of $L$, we have Galois representations 
  \[ 
   \rho_{f,\pp},\, \rho_{g,\pp}:\, G_{\QQ}\rightarrow \GL_2(\mathcal{O}_{L, \pp}).
  \]
  Let $V_{f,g,\mathfrak{p}}$ be the $4$-dimensional $L_{\mathfrak{p}}$-vector space on which $G_{\QQ}$ acts with $\rho_{f,\mathfrak{p}}\otimes \rho_{g,\mathfrak{p}}$, and let $V=V^*_{f,g,\mathfrak{p}}$, and $T \subset V$ the natural $\cO_{L, \mathfrak{p}}$-lattice. Let $k$ denote the residue field of $\mathcal{O}_{L_{\mathfrak{p}}}$. 
  
  We shall suppose that the triple $(f, g, \pp)$ satisfies ``Hypothesis (BI)'' from \S 11.1 of \cite{KLZ17}. As shown in \emph{op.cit.}, this can only hold if the characters of $f$ and $g$ satisfy $\varepsilon_f \varepsilon_g \ne 1$; but if this is true, for a given $(f, g)$, then Hypothesis (BI) typically holds for a large proportion of primes $\pp$ of the coefficient field, and frequently for finitely many.
  
  \begin{remark}
   For example, Hypothesis (BI) holds for all but finitely many $\pp$ in either of the following two situations (see \cite[Remark 11.1.3]{KLZ17}):
   \begin{itemize}
   \item $r, r' \ge 2$, neither $f$ nor $g$ is of CM-type, the levels $(N_f, N_g)$ are coprime, and at least one of $r, r'$ is odd.
   \item $r \ge 2$, $f$ is not of CM-type, $f$ is not a twist of $g$, $r' = 1$, and $(N_f, N_g) = 1$.\qedhere
   \end{itemize}
   More general, but more complicated, criteria in terms of the ``inner twists'' of $f$ and $g$ can be found in \cite{loeffler17} and \cite{studnia24}.
  \end{remark}
  
  \begin{proposition}
   The representation $T$ satisfies the hypotheses (Hyp1)--(Hyp3).
  \end{proposition}
  
  \begin{proof}
   Hypotheses (Hyp1a) and (Hyp2) are explicitly included as part of Hypothesis (BI) of \cite{KLZ17}. The remaining part of Hypothesis (BI) asserts the existence of an element $\sigma \in \Gal(\overline{\QQ} / \QQ(\mu_{p^\infty}))$ acting as $-1$ on $T$. Clearly this implies (Hyp1b), since we may take $\gamma = \sigma$. For (Hyp3), if $\Omega$ denotes the smallest extension of $\QQ$ whose Galois group acts trivially on $T$ and on $\mu_{p^\infty}$, then $\sigma$ defines an element of the centre of $\Gal(\Omega / \QQ)$ which acts as $-1$ on $T$; as in \cite[Prop 7.2.20]{leiloefflerzerbes14} this implies the vanishing of the groups $H^1(\Omega/\QQ, V/T)$ and $H^1(\Omega/\QQ, (V/T)^*(1))$. 
  \end{proof}

 \subsection{Triangulations} 
  
  For $\star\in\{f,g\}$, write $\alpha_\star, \beta_\star$ for the roots of the Hecke polynomial at $p$.  Then the restriction of $V$ to $G_{\QQ_p}$ is crystalline with Hodge--Tate weights $\{0, r'-1, r-1, r + r' - 2 \}$, and the eigenvalues of $\varphi^{-1}$ on $\Dcris(V)$ are the pairwise products of the roots of the Hecke polynomials
  \begin{equation}\label{eq:prodHecke}
   \{\alpha_f\alpha_g,\, \alpha_f\beta_g,\, \beta_f\alpha_g,\, \beta_f\beta_g\}.
  \end{equation}We make the following assumptions:

  \begin{assumption}\label{ass:Heckeevals}\
   \begin{enumerate}
    \item ($p$-regularity) We have $\alpha_f\neq \beta_f$ and $\alpha_g\neq \beta_g$.
    \item (Non-critical slope) We have
     \begin{align*}
      v_p(\alpha_f) & < r - 1\\
      v_p(\alpha_g) & < r' - 1 \quad \text{unless $r' = 1$}
     \end{align*}
     \item (No local zeros, c.f. \S 8.1 in \cite{loefflerzerbes16}) None of the elements of \eqref{eq:prodHecke} are a power of $p$.
   \end{enumerate}
  \end{assumption}
  
  (Note that the elements \eqref{eq:prodHecke} all have absolute value $p^{(r_1 + r_2 - 2)/2}$, so the ``no local zeros'' condition is immediate if $r_1 + r_2$ is odd.)

  \begin{proposition}[c.f. {\cite[Theorem 6.3.2]{loefflerzerbes16}}]
   Let $\star\in\{f,g\}$. Then the $(\varphi, \Gamma)$-module $\cD_\star = \DD^{\dag}_{\rig}(V_p(\star)^*)$ admits a canonical triangulation
   \[
   0\rightarrow \mathscr{F}^+\cD_\star\rightarrow \cD_\star\rightarrow \mathscr{F}^-\cD_\star\rightarrow 0,
   \]
   with $\mathscr{F}^-\cD_\star$ having Hodge--Tate weight 0 and crystalline Frobenius eigenvalue $\alpha_\star^{-1}$.\qed
  \end{proposition}

  \begin{definition}
   Define the $(\varphi,\Gamma)$-submodules $\cD^{++} \subset \cD^+ \subset \cD_f \otimes \cD_g$ by 
   \begin{align*}
    \cD^+ &\coloneqq \left(\mathscr{F}^+\cD_f\otimes \cD_g\right)+\left(\cD_f\otimes \mathscr{F}^+\cD_g\right)\, \subset\, \cD_f\otimes \cD_g, \\
    \cD^{++} & \coloneqq \mathscr{F}^+\cD_f \otimes \cD_g,
   \end{align*}
   so that $\cD^+$ has rank $3$, $\cD^{++}$ rank 2, and $\cD^+ / \cD^{++} \cong \mathscr{F}^-\cD_f \otimes \mathscr{F}^+\cD_g$.
  \end{definition}
  
  (The submodule $\cD^+$ will be the local condition for our Euler system; the smaller submodule $\cD^{++}$ should be related to $p$-adic $L$-functions.)
  
 \subsection{The $p$-adic $L$-function} We recall the following theorem:
 
  \begin{theorem}\label{thm:padicL}
   There exists an element $L_{p, \alpha_f}(f, g) \in \cH(\Gamma)$ with the following properties:
   \begin{enumerate}[(i)]
    \item For each locally-algebraic character $j + \chi$ of $\Gamma$ with $0 \le j \le r - r' - 1$, we have an interpolation formula
    \[ L_{p, \alpha_f}(f, g)(j + \chi) = (\star) \cdot \frac{L(f, g, \bar{\chi}, 1 + j)}{\langle f, f \rangle} \]
    where $(\star)$ is an explicit factor, and $\langle f, f \rangle$ is the Petersson norm.
    \item If $\cF$ and $\cG$ are Coleman families (over some small affinoids $U_f$, $U_g$) passing through the $p$-stablizations of $f$, $g$ corresponding to $\alpha_f$ and $\alpha_g$, then $L_{p, \alpha_f}(f, g)$ extends to a distribution $L_p(\cF, \cG)$ valued in $\cO(U_f \times U_g)$ with the appropriate interpolation property at classical points.
   \end{enumerate}
  \end{theorem}
  
  \begin{proof}
   If $\alpha_f$ and $\alpha_g$ are $p$-adic units, then all assertions of the theorem are contained in works of Hida and Panchishkin from the 1980's (see in particular \cite{hida88}). In the case of non-ordinary families, we refer to \S 7.3 of \cite{grahampillonirodrigues}. (Note that the $p$-adic $L$-function of \emph{op.cit.} differs from ours by a shift in the cyclotomic variable, since we have chosen to maintain the normalisations used in \cite{loefflerzerbes16}.)
  \end{proof}
  
 \subsection{The Euler system}

  Let $c > 1$ be coprime to $6pN_f N_g$, and we let $\Sigma$ be a finite set of primes containing all those dividing $c p N_f N_g $.

  \begin{proposition}[{\cite[Theorem 8.1.4]{loefflerzerbes16}}]
   There exists an analytic Euler system $(c_m)_{m\geq 1,\, (m,pc=1)}$ for $V$, in the sense of \cref{def:anES}, with $c_m\in H^1_{\an,\Sigma}(\QQ_\infty(m),V)$. Moreover,  the Euler system $(c_m)$ satisfies the $\cD^+$ local condition at $p$.
  \end{proposition}
  \begin{proof}
  This is a restatement of Theorem 7.1.2 of \cite{loefflerzerbes16}.
 \end{proof}

 \begin{remark}
  Note that since $V$ is 4-dimensional, a rank 3 local-condition module is the correct rank for \cref{thm:mainSelbound} to apply. This submodule $\cD^+$ arises from a subrepresentation if, and only if, $f$ and $g$ are both ordinary and $\alpha_f, \alpha_g$ are the unit roots of the Hecke polynomials, which is the setting studied in \cite{KLZ17}.
 \end{remark}

 \begin{theorem}
  The image of $c_1$ under the map
  \[ \widetilde{H}^1_{\mathrm{an}}(\QQ_\infty, V, \cD^+) \to H^1_{\mathrm{an}}(\QQ_{p, \infty}, \cD^+ / \cD^{++}) \to \cH(\Gamma), \]
  where the third map is the Perrin-Riou regulator for the one-dimensional $(\varphi, \Gamma)$-module $\cD^+ / \cD^{++}$, sends $c_1$ to the $p$-adic $L$-function $L_{p, \alpha_f}(f, g)$.
 \end{theorem}
 
 \begin{proof}
  This explicit reciprocity law was proved in \cite{loefflerzerbes16} with the reservation that the $p$-adic $L$-function used in \emph{op.cit.} was only known to satisfy the interpolation property (i) of \cref{thm:padicL} when $\chi = 1$.  However, since these $\chi = 1$ specialisations are sufficient to uniquely determine the 3-variable $p$-adic $L$-function once $f$ and $g$ are allowed to varying in families, this $p$-adic $L$-function must agree with the one constructed in \cite{grahampillonirodrigues}.
 \end{proof}
 
 \begin{corollary}
  Assume that $r-r'\geq 2$. Then $L_{p, \alpha_f}(f, g)$ is not a zero-divisor, and hence $e_\eta c_1 \neq 0$ for all $\eta$.
 \end{corollary}
 
 \begin{proof}
  See Remark 11.6.5 in \cite{KLZ17}; although ordinarity is assumed in \emph{op.cit.}, the same argument also applies in the non-ordinary case.
 \end{proof}

 Applying Theorem \ref{thm:H2tors}, we obtain the following result:

 \begin{theorem}\label{thm:RSIMC}
  If $r - r' = 1$, then assume $L_{p, \alpha_f}(f, g)$ is not a zero-divisor. Then $\widetilde{H}^2_{\an}(\QQ_\infty, V, \cD^+)$ is $\cH(\Gamma)$-torsion, and
  \[ \mathrm{rk}_{\cH(\Gamma)}\widetilde{H}^1_{\an}(\QQ_\infty, V, \cD^+)=1.\]
  Moreover, we have the divisibilities of characteristic ideals
  \[
  \operatorname{char}_{\cH(\Gamma)} \widetilde{H}^2_{\mathrm{an}}(\QQ_\infty, V, \cD^+) \mathrel{\Big|} \operatorname{char}_{\cH(\Gamma)} \left(\frac{\widetilde{H}^1_{\mathrm{an}}(\QQ_\infty, V, \cD^+)}{\cH(\Gamma) \cdot c_1}\right)
  \]
  and
  \[
   \operatorname{char}_{\cH(\Gamma)} \widetilde{H}^2(\QQ_\infty, V, \cD^{++}) \mathrel{\Big|} L_{p, \alpha_f}(f, g),
  \]
  and equality holds in the first divisibility iff it holds in the second.
 \end{theorem}
 
 \begin{proof}
  The rank statement and the first divisibility are precisely the result of \cref{thm:mainSelbound} in the present setting; the assumption \ref{ass:Heckeevals} (3) implies that the ``error terms'' $\cE^1$ and $\cE^2$ are trivial.
  
  The second divisibility (and the condition for equality) follows from the first using the exact triangle
  \[ \widetilde{R\Gamma}(\QQ_\infty, V, \cD^{++}) \to \widetilde{R\Gamma}(\QQ_\infty, V, \cD^+) \to R\Gamma_{\mathrm{an}}(\QQ_{p, \infty}, \cD^+ / \cD^{++}) \to [+1]\]
  and the trivialization of the third complex given by the Perrin-Riou regulator map (exactly as in \S 11.6 of \cite{KLZ17} in the ordinary case); again, the ``no extra zero'' assumption  assumption \ref{ass:Heckeevals} (3) shows that the regulator map is a bijection.
 \end{proof}

\subsection{Generalisations} We briefly sketch some other cases in which the same methods can be applied (full details will appear elsewhere).
 
 \subsubsection*{Asai $L$-functions} Let $F$ be a real quadratic field, and let $p > 3$ be a prime which splits in $F$. Let $\cF$ be a quadratic Hilbert modular form of weight $(k_1,k_2)$ with $k_1 > k_2 \ge 2$ and level coprime to $p$, and take for $V$ the dual of the Asai Galois representation associated to $\cF$ (the tensor induction of the standard 2-dimensional representation). We fix choices of roots $\alpha_1, \alpha_2$ of the Hecke polynomials at the primes above $p$ (which are the analogues of the $\alpha_f, \alpha_g$ above), which we assume to be sufficiently small, i.e.~$v_p(\alpha_i) < k_i - 1$. 
 
 In this case, $V$ is again four-dimensional, odd, and crystalline at $p$; and the choice of the $\alpha_i$ determines submodules $\cD^{++} \subset \cD^+ \subset \DD^{\dag}_{\rig}(V)$ of ranks 2 and 3 respectively. The Euler system constructed in \cite{leiloefflerzerbes18} in the ordinary case can be extended to this more general finite-slope setting, and gives an analytic Euler system for $V$ satisfying the local condition $\cD^{+}$. Details of this construction will appear in the PhD thesis of Ana Marija Vego at ETH Z\"urich.
 
 In this finite-slope setting, a $p$-adic Asai $L$-function was constructed by Kazi and the first author in \cite{kaziloeffler}. Generalising the results in \cite{grossiloefflerzerbes-GO4} for the ordinary case, it should be accessible to show that the image of this Euler system under the Perrin-Riou regulator for $\cD^+ / \cD^{++}$ is the $p$-adic $L$-function; and following the strategy above will then give one inclusion in the Main Conjecture in this setting.
 
 \subsubsection*{The $\operatorname{GSp}_4$ case} Let $\Pi$ be a cohomological, cuspidal automorphic representation of $\GSp_4$ which is generic, and is not CAP or endoscopic, and such that $\Pi_p$ is unramified Then we have a 4-dimensional Galois representation $V = V_p(\Pi)^*$ associated to $\Pi$, which is crystalline at $p$.
 
 If we fix a $p$-refinement of $\Pi_p$ (a choice of one of the Weyl-group orbits of characters of the diagonal torus from which $\Pi_p$ is induced) then this determines, as in the previous cases, a pair of submodules $\cD^{++} \subset \cD^+ \subset \DD^{\dag}_{\rig}(V)$ of ranks 2 and 3 respectively. Under a mild ``non-critical slope'' assumption on the $p$-refinement, Rockwood has constructed in \cite{rockwood26} an analytic Euler system for $V$ (generalising the construction of \cite{LSZ17} in the ordinary case). Via very similar arguments as used in \cite{loefflerzerbes16} in the Rankin--Selberg setting, one can check that Rockwood's Euler system satifies the $\cD^+$ local condition at $p$, giving another natural example in which the hypotheses of \cref{thm:mainSelbound} are satisfied.
 
 On the ``analytic'' side, one can construct a $p$-adic $L$-function for $\Pi$ using higher Coleman theory, as explained in \cite{LZ21-BSD}. (As explained in \cite[\S 6]{LZ21-BSD}, the main theorems of \emph{op.cit.} assume ordinarity, but this is not needed for the $p$-adic $L$-function construction; in fact the only reason why ordinarity was assumed in \emph{op.cit.} was precisely because we did not know at the time if it was possible to bound Pottharst-style Selmer groups via Euler systems.) Thus the methods of the present paper should allow one to prove the Euler-system divisibility of the Iwasawa main conjecture for $\GSp_4$ at non-ordinary primes. 
 
 \begin{remark}
  It also seems likely that these methods could be used to bound Selmer groups for the 8-dimensional Galois representations associated to $\GSp_4 \times \GL_2$ at non-ordinary primes, using the Euler system of \cite{HJS20} and the $p$-adic $L$-functions constructed in \cite{grahamrockwood24}.
 \end{remark}
 
 \newcommand{\noopsort}[1]{\relax}
 \providecommand{\bysame}{\leavevmode\hbox to3em{\hrulefill}\thinspace}
 \providecommand{\MR}[1]{%
  MR \href{http://www.ams.org/mathscinet-getitem?mr=#1}{#1}.
 }
 \providecommand{\href}[2]{#2}
 \newcommand{\articlehref}[2]{\href{#1}{#2}}


\end{document}